\documentclass[12pt]{amsart}

\usepackage{amssymb}
\usepackage{bm}
\usepackage{graphicx}
\usepackage{enumerate}
\usepackage{psfrag}
\usepackage{color}
\usepackage{url}

\newcommand{\excise}[1]{}

\newtheorem{thm}{Theorem}[section]
\newtheorem{lemma}[thm]{Lemma}

\newtheorem{cor}[thm]{Corollary}
\newtheorem{prop}[thm]{Proposition}

\newtheorem{prob}[thm]{Problem}

\theoremstyle{definition}
\newtheorem{example}[thm]{Example}
\newtheorem{remark}[thm]{Remark}
\newtheorem{defn}[thm]{Definition}
\newtheorem{conv}[thm]{Convention}

\numberwithin{equation}{section}

\newcommand{\ring}[1]{\ensuremath{\mathbb{#1}}}

\renewcommand\>{\rangle}
\newcommand\<{\langle}

\newcommand\NN{\ring{N}}

\newcommand\QQ{\ring{Q}}

\newcommand\ZZ{\ring{Z}}

\newcommand\kk{\Bbbk}

\newcommand\xx{{\mathbf x}}
\newcommand\yy{{\mathbf y}}
\newcommand\bb{{\mathbf b}}
\newcommand\cc{{\mathbf c}}

\renewcommand\aa{{\mathbf a}}

\newcommand\iso{\cong}

\DeclareMathOperator\spann{span} 

\DeclareMathOperator\lcm{lcm} 

\begin{document}

\mbox{}
\title{On factorization invariants and Hilbert functions\qquad}
\author{Christopher O'Neill}
\address{Mathematics\\University of California, Davis\\One Shields Ave\\Davis, CA 95616}
\email{coneill@math.ucdavis.edu}

\date{\today}

\begin{abstract}
\hspace{-2.05032pt}
Nonunique factorization in cancellative commutative semigroups is often studied using combinatorial factorization invariants, which assign to each semigroup element a quantity determined by the factorization structure.  For numerical semigroups (additive subsemigroups of the natural numbers), several factorization invariants are known to admit predictable behavior for sufficiently large semigroup elements.  In particular, the catenary degree and delta set invariants are both eventually periodic, and the omega-primality invariant is eventually quasilinear.  In this paper, we demonstrate how each of these invariants is determined by Hilbert functions of graded modules.  In doing so, we extend each of the aforementioned eventual behavior results to finitely generated semigroups, and provide a new framework through which to study factorization structures in this setting.  
\end{abstract}
\maketitle



\section{Introduction}\label{s:intro}

A \emph{factorization} of an element $\alpha$ of a cancellative commutative semigroup $(\Gamma,+)$ is an expression of $\alpha$ as a sum of irreducible elements of $\Gamma$, and a \emph{factorization invariant} is a quantity assigned to each element of $\Gamma$ (or to $\Gamma$ as a whole) that measures the failure of its factorizations to be unique.  Factorization invariants are often combinatorial in nature, and provide concrete methods of quantifying the abundance and variety of factorizations.  For instance, one may consider the number of distinct factorizations of an element $\alpha \in \Gamma$, or the maximum number of irreducible elements appearing in a single factorization of $\alpha$.  See \cite{nonuniq} for a thorough introduction.  


Several recent results examine the asymptotic behavior of factorization invariants in the setting of numerical semigroups (additive, cofinite subsemigroups of $\NN$).  For example, the delta set (Definition~\ref{d:deltaset}) and catenary degree (Definition~\ref{d:catenary}) invariants, which measure the ``spread'' of a given element's nonunique factorizations, are each eventually periodic over any numerical semigroup \cite{catenaryperiodic,deltaperiodic}.  Additionally, the $\omega$-primality invariant (Definition~\ref{d:omega}), which assigns a positive integer to each semigroup element measuring how far from prime that element is, coincides with a linear function with periodic coeffients for sufficiently large elements in any numerical semigroup \cite{omegaquasi}.  See the survey~\cite{numericalsurvey} and the references therein for more detail on this setting.  

The primary goals of this paper are to (i) generalize each result in the previous paragraph to the setting of finitely generated semigroups using techniques from combinatorial commutative algebra, and in doing so, (ii) provide a new framework through which to study these invariants.  Given a finitely generated, reduced, cancellative commutative semigroup $\Gamma$, we construct, for each factorization invariant discussed above, a family of multigraded modules whose Hilbert functions (Definition~\ref{d:hilbert}) determine the value of the invariant in question for any element of $\Gamma$.  Applying Hilbert's theorem (Theorems~\ref{t:hilbert} and~\ref{t:hilbert2}) to each family of modules classifies the eventual behavior of the corresponding factorization invariant in $\Gamma$ (Theorems~\ref{t:deltahilbert}, \ref{t:omegahilbert}, and~\ref{t:catenaryhilbert}).  In the special case where $\Gamma \subset \NN$ is a numerical semigroup, each classification specializes to a result from the previous paragraph (Corollaries~\ref{c:deltanumerical}, \ref{c:omeganumerical} and~\ref{c:catenarynumerical}).  

In contrast to the semigroup-theoretic arguments originally used for numerical semigroups, the arguments presented here lie squarely in the realm of combinatorial commutative algebra.  As such, our approach provides new theoretical tools with which to study factorization invariants in this setting.  This includes several classes of semigroups of interest in factorization theory, such as Cohen-Kaplansky domains (integral domains with finitely many irreducible elements), which are of interest in algebraic number theory~\cite{ckdomains,mottdivtheory}, and block monoids, which are central to additive combinatorics~\cite{block}.  In~fact, questions arising in algebraic number theory motivated the initial study the $\omega$-primality invariant~\cite{origin,tame}.  Additionally, in the setting of affine semigroups (finitely generated subsemigroups of~$\NN^d$), several invariants discussed in this paper are of interest outside of factorization theory.  Indeed, factorizations of an affine semigroup element coincide with integer solutions to a system of linear Diophantine equations.  From this viewpoint, delta sets of affine semigroup elements are closely related to questions of lattice width \cite{computelatticewidth,latticewidthplanarsets}, and catenary degree computations encapsulate data related to $\ell_1$-distances between integer solutions~\cite{combopt}.  

The paper is organized as follows.  In Section~\ref{s:hilbert}, we review Hilbert's theorem, both for $\NN$-graded modules (Theorem~\ref{t:hilbert}) and multigraded modules (Theorem~\ref{t:hilbert2}).  We also review multivariate quasipolynomial functions, including several equivalent definitions (Theorem~\ref{t:affineequiv}).  
The remaining sections of the paper consider different factorization invariants for finitely generated semigroups, including the number of distinct factorizations (Section~\ref{s:factorizations}), the delta set (Section~\ref{s:delta}), $\omega$-primality (Section~\ref{s:omega}), and the catenary degree (Section~\ref{s:catenary}).  
We demonstrate how the value of each invariant can be recovered from Hilbert functions, and examine consequences both for finitely generated semigroups and for numerical semigroups.  


\section{Hilbert functions of multigraded modules}\label{s:hilbert}

In this section, we survey the definitions and results from combinatorial commutative algebra that will be used throughout this paper.  See \cite{cca} for a thorough introduction.  

\begin{conv}\label{con:names}
Throughout this paper, we denote by $\kk$ an arbitrary field, $T$ a finite Abelian group, $d \ge 1$ a positive integer, and $A = \NN^d \oplus T$.  Additionally, given $a = (a_1, \ldots, a_k) \in \NN^k$, we write $\xx^\aa$ for the monomial $x_1^{a_1} \cdots x_k^{a_k}$ in the polynomial ring $\kk[x_1, \ldots, x_k]$.  Lastly, let $\NN = \{0, 1, 2, \ldots\}$.  
\end{conv}

\begin{defn}\label{d:hilbert}
Fix a $\kk$-algebra $S$.  An \emph{$A$-grading of $S$} is an expression 
$$S \iso \bigoplus_{\alpha \in A} S_\alpha$$
of $S$ as a direct sum of finite dimensional $\kk$-subspaces of $S$, indexed by $A$, such that $S_\alpha S_\beta \subset S_{\alpha + \beta}$ for all $\alpha, \beta \in A$.  
An \emph{$A$-grading of a module $M$ over $S$} is an expression 
$$M \iso \bigoplus_{\alpha \in A} M_\alpha$$
of $M$ as a direct sum of $\kk$-subspaces of $M$, indexed by $A$, with $S_\alpha M_\beta \subset M_{\alpha + \beta}$ for all $\alpha, \beta \in A$.  Such a grading is \emph{modest} if $\dim_\kk M_\alpha < \infty$ for all $\alpha \in A$.  The \emph{Hilbert function of a modestly $A$-graded $S$-module $M$} is the function $\mathcal H(M;-):A \to \ZZ_{\ge 0}$ given by
$$\mathcal H(M;\alpha) = \dim_\kk M_\alpha$$
for each $\alpha \in A$.  
\end{defn}

Theorem~\ref{t:hilbert}, whose original form is due to Hilbert, characterizes the eventual behavior of the Hilbert functions of certain $\NN$-graded modules.  

\begin{defn}\label{d:quasipolynomial}
A function $f:\NN \to \QQ$ is a \emph{quasipolynomial of degree $k$} if there exist periodic functions $a_0, \ldots, a_k:\NN \to \QQ$ such that
$$f(n) = a_k(n)n^k + \cdots + a_1(n)n + a_0(n)$$
and $a_k$ is not identically zero.  The \emph{period of $f$} is the minimal positive integer $\pi$ such that $a_i(n + \pi) = a_i(n)$ for all $i \le k$ and $n \in \NN$.  
\end{defn}

\begin{thm}[Hilbert]\label{t:hilbert}
Fix an $\NN$-graded $\kk$-algebra $S$, and a finitely generated, graded $S$-module $M$ of dimension $d$.  For $n \gg 0$, the Hilbert function of $M$ coincides with a quasipolynomial of degree $d-1$ (called the \emph{Hilbert quasipolynomial of $M$}).  More specifically, there exist periodic functions $a_0, \ldots, a_{d-1}:\NN \to \QQ$ such that $a_{d-1} \not\equiv 0$ and
$$\mathcal H(M;n) = a_{d-1}(n)n^{d-1} + \cdots + a_1(n)n + a_0(n)$$
for sufficiently large $n$.  Additionally, if $y_1, \ldots, y_d \in S$ is a homogeneous system of parameters for $M$, then the period of each $a_i$ divides $\lcm(\deg(y_1), \ldots, \deg(y_d))$.  
\end{thm}

The following result, due to Bruns and Ichim \cite{quasicoeffs}, yields more control over the periods of the coefficients of the Hilbert quasipolynomial in Theorem~\ref{t:hilbert}.  

\begin{thm}[{\cite[Theorem~2]{quasicoeffs}}]\label{t:constcoeffs}
Fix an $\NN$-graded $\kk$-algebra $S$, and an $\NN$-graded $S$-module $M$ of dimension $d$.  Fix $a_0, \ldots, a_{d-1}:\NN \to \QQ$ periodic such that
$$f(n) = a_{d-1}(n)n^{d-1} + \cdots + a_1(n)n + a_0(n)$$
is the Hilbert quasipolynomial of $M$, and suppose $f$ has period $\pi$.  The coefficient $a_i$ is constant for all $i \ge \dim M/IM$, where $I = \<x \in R : \gcd(\pi,\deg(x)) = 1\>$.  
\end{thm}

We conclude this section with Theorem~\ref{t:hilbert2}, a generalization of Hilbert's theorem to modest $A$-gradings.  First, we define multivariate quasipolynomials on $A$ (Definition~\ref{d:eventualquasipolynomial}) and give several equivalent definitions in Theorem~\ref{t:affineequiv}.  

\begin{remark}\label{r:affine}
Fields \cite{multiquasi} gives a thorough and detailed introduction to multivariate quasipolynomials in the case $A = \NN^d$, including proofs ``from scratch'' of some portions of Theorem~\ref{t:affineequiv}.  Most of the proofs immediately generalize to our setting (where $A$ may have torsion), so in what follows we give only the most relevant definitions and results.  The interested reader is encouraged to consult \cite{multiquasi}.  Lemma~\ref{l:projection} is the key to generalizing from the case where $A$ is torsion-free, and in particular ensures the polynomial restrictions in Definition~\ref{d:eventualquasipolynomial}(b) are well-defined.  
\end{remark}

\begin{lemma}\label{l:projection}
Let $\rho:A \to \NN^d$ denote the projection map.  Elements $\alpha_1, \ldots, \alpha_r \subset A$ are linearly independent if and only if their projections $\rho(\alpha_1), \ldots, \rho(\alpha_r)$  are linearly independent.  Moreover, the restriction of $\rho$ to $\NN\alpha_1 + \cdots + \NN\alpha_r \subset A$ is a bijection.  
\end{lemma}

\begin{defn}\label{d:eventualquasipolynomial}
Fix $f:A \to \QQ$, linearly independent $\alpha_1, \ldots, \alpha_r \in A$, and $\beta \in A$.  
\begin{enumerate}[(a)]
\item 
The \emph{cone generated by $\alpha_1, \ldots, \alpha_r$ translated by $\beta$} is the set
$$C = C(\beta;\alpha_1, \ldots, \alpha_r) = \left\{\beta + \textstyle\sum_{j = 1}^r c_j\alpha_j : c_1, \ldots, c_r \in \NN\right\} \subset A.$$

\item 
The function $f$ is a \emph{simple quasipolynomial supported on a cone $C$} if (i) $f$ vanishes outside of $C$ and (ii) $f$ coincides with a polynomial $p$ when restricted to $C$ and projected onto $\NN^d$ (in the sense of Lemma~\ref{l:projection}).  The \emph{degree of $f$}, denoted $\deg(f)$, is the smallest possible degree for $p$, and the \emph{cumulative degree of~$f$} is $r + \deg(f)$.  


\item 
The function $f$ is \emph{eventually quasipolynomial} if it is a finite sum of simple quasipolynomials.  The \emph{cumulative degree} of $f$ is the minimal integer $k$ such that $f$ can be written as a finite sum of simple quasipolynomials of cumulative degree at most $k$.  

\end{enumerate}
\end{defn}

\begin{remark}\label{r:eventualquasipolynomial}
The terminology in Definition~\ref{d:eventualquasipolynomial} differs slightly from \cite{multiquasi}, where the term ``quasipolynomial'' is used in place of ``eventually quasipolynomial''.  However, Definition~\ref{d:eventualquasipolynomial} was chosen so that ``eventual quasipolynomial'' coincides with Definition~\ref{d:quasipolynomial} when $A = \NN$.  Example~\ref{e:eventualquasipolynomial} discusses this case in more detail.  
\end{remark}

\begin{thm}\label{t:affineequiv}
Given a function $f:A \to \QQ$, the following are equivalent.  
\begin{enumerate}[(a)]
\item 
The function $f$ is eventually quasipolynomial.  

\item 
There exists a finite collection of cones $C_1, \ldots, C_k \subset A$ such that $A = \bigcup_i C_i$ and $f$ coincides with a polynomial when restricted to each $C_i$.  

\item 
The function $f$ is a sum of simple quasipolynomials with disjoint support.  

\item 
Writing $A = \ZZ_{d_1} \oplus \cdots \oplus \ZZ_{d_m} \oplus \NN^d$ for $d_1, \ldots, d_m \in \ZZ_{> 1}$ and 
$$\QQ[\![A]\!] = \QQ[\![x_1, \ldots, x_{d+m}]\!]/\<x_i^{d_i} - 1 : 1 \le i \le m\>$$
for the formal power series ring over $A$ with rational coefficients, the generating function $F(\xx) \in \QQ[\![A]\!]$ for $f$ has the form 
$$F(\xx) = \sum_{\alpha \in A} f(\alpha)\xx^\alpha = \frac{P(\xx)}{\prod_{j=1}^r (1 - \xx^{\alpha_j})} \in \QQ[\![A]\!]$$
for some $\alpha_1, \ldots, \alpha_r \in A$ and $P \in \QQ[A]$.  

\end{enumerate}
\end{thm}

\begin{proof}
If $A = \NN^d$, then \cite[Theorem~26]{multiquasi} proves the equivalence of (a) and (d), and \cite[Theorems~2 and~12]{khov2} prove the remaining equivalences.  The proof for general $A$ is identical to those given in the aforementioned references.  
\end{proof}

\begin{example}\label{e:eventualquasipolynomial}
Suppose $f:\NN \to \QQ$ is eventually quasipolynomial in the sense of Definition~\ref{d:eventualquasipolynomial}.  Since any cone in $\NN$ has dimension at most 1, Theorem~\ref{t:affineequiv}(c) implies there exist disjoint 1-dimensional cones $C_1, \ldots, C_k$ whose union contains all but finitely many elements of $\NN$ in such a way that $f$ coincides with a polynomial when restricted to each~$C_i$.  Each element of $\NN \setminus \bigcup_i C_i$ corresponds to a 0-dimensional cone.  This means $f$ coincides with a quasipolynomial (in the sense of Definition~\ref{d:quasipolynomial}) for all $n \in \bigcup_i C_i$, and the period of $f$ divides the least common multiple of the generators of $C_1, \ldots, C_k$.  
\end{example}

There are several concrete examples of eventually quasipolynomial functions in the later sections of this paper.  For instance, see Examples~\ref{e:factset} and~\ref{e:deltaaffine}.  

\begin{thm}\label{t:hilbert2}
Fix an $A$-graded $\kk$-algebra $S$, and a finitely generated, modestly graded $S$-module $M$.  The Hilbert function $\mathcal H(M;-)$ of $M$ is eventually quasipolynomial of cumulative degree $\dim M$.  
\end{thm}

\begin{proof}
Apply Theorem~\ref{t:affineequiv} and \cite[Theorem~8.41]{cca}.  
\end{proof}

\begin{remark}\label{r:modestgrading}
Throughout the remainder of this paper, all $\kk$-algebra and module gradings will be modest, and as such this word is often omitted.  See \cite[Section~8.4]{cca} for a thorough discussion of modest gradings.  
\end{remark}

\section{The number of distinct factorizations}\label{s:factorizations}

After introducing some notation for factorizations (Definition~\ref{d:factset}) in the context of finitely generated semigroups and numerical semigroups (Definition~\ref{d:affine}), we examine the number of distinct factorizations of semigroup elements.  The main result of this section is Proposition~\ref{p:factorhilbert}, which presents the connection between Hilbert functions and factorization invariants on which the rest of this paper is based.  As a direct consequence, we recover an alternative proof of an asymptotics result from the literature (Theorem~\ref{t:factorasymp}) and its specialization to numerical semigroups (Corollary~\ref{c:factornumerical}).  

\begin{conv}\label{con:semigroups}
Throughout the rest of this paper, $\Gamma = \<\alpha_1, \ldots, \alpha_r\> \subset A$ denotes a finitely generated, reduced subsemigroup of $A$.  Whenever we write $\Gamma = \<\alpha_1, \ldots, \alpha_r\>$, we assume the elements $\alpha_1, \ldots, \alpha_r$ comprise the (unique) minimal generating set for $\Gamma$.  
\end{conv}

\begin{defn}\label{d:affine}
A semigroup $\Gamma$ is \emph{affine} if $\Gamma \subset \NN^d$.  If $\Gamma \subset \NN$ and $\gcd(\Gamma) = 1$, we say $\Gamma$ is a \emph{numerical semigroup}.  
\end{defn}

\begin{defn}\label{d:factset}
Fix a finitely generated semigroup $\Gamma = \<\alpha_1, \ldots, \alpha_r\> \subset A$.  The elements $\alpha_1, \ldots, \alpha_r$ comprising the unique minimal generating set of $\Gamma$ are called \emph{irreducible} (or \emph{atoms}).  A \emph{factorization} of $\alpha \in \Gamma$ is an expression 
$$\alpha = a_1\alpha_1 + \cdots + a_r\alpha_r$$
of $\alpha$ as a finite sum of atoms, which we denote by the $r$-tuple $\aa = (a_1, \ldots, a_r) \in \NN^r$.  Write $\mathsf Z_\Gamma(\alpha)$ for the \emph{set of factorizations} of an element $\alpha \in \Gamma$, viewed as a subset of~$\NN^r$.  If $\Gamma$ is a numerical semigroup, we assume $\alpha_1 < \cdots < \alpha_r$.  
\end{defn}


\begin{example}\label{e:factset}
Consider the affine semigroup $\Gamma = \<(2,1),(1,1),(1,2)\> \subset \NN^2$.  Restricting $|\mathsf Z_\Gamma(-)|$ to the cone $C((0,0);(2,1),(3,3))$ yields a simple quasipolynomial of degree~1 given by $|\mathsf Z_\Gamma(x,y)| = -\frac{1}{3}x + \frac{2}{3}y + 1$.  In fact, the union of the six cones below equals~$\Gamma$, and restricting $|\mathsf Z_\Gamma(-)|$ to each cone yields a simple quasilinear function.  The cones in the first row are depicted in Figure~\ref{f:factset}, and those in the second row are reflections of those in the first row about the line $y = x$.  

\begin{center}
\begin{tabular}{c@{\hspace{0.4in}}c@{\hspace{0.4in}}c}
$C((0,0);(2,1),(3,3))$ & $C((1,1);(2,1),(3,3))$ & $C((2,2);(2,1),(3,3))$ \\
$C((0,0);(1,2),(3,3))$ & $C((1,1);(1,2),(3,3))$ & $C((2,2);(1,2),(3,3))$ \\
\end{tabular}
\end{center}

This demonstrates that $|\mathsf Z_\Gamma(-)|$ is eventually quasilinear by Theorem~\ref{t:affineequiv}(b).  One can also express $|\mathsf Z_\Gamma(-)|$ as the sum of these six simple quasilinear functions minus the restriction of $|\mathsf Z_\Gamma(-)|$ to each nonempty intersection therein, each of which is a translation of $C((0,0);(3,3))$.  The existence of both of these expressions is ensured by Proposition~\ref{p:factorhilbert}, and Remark~\ref{r:factset} explains how each function may be computed.  
\end{example}

\begin{figure}
\begin{tabular}{c@{\hspace{0.2in}}c@{\hspace{0.2in}}c}

\includegraphics[width=1.5in]{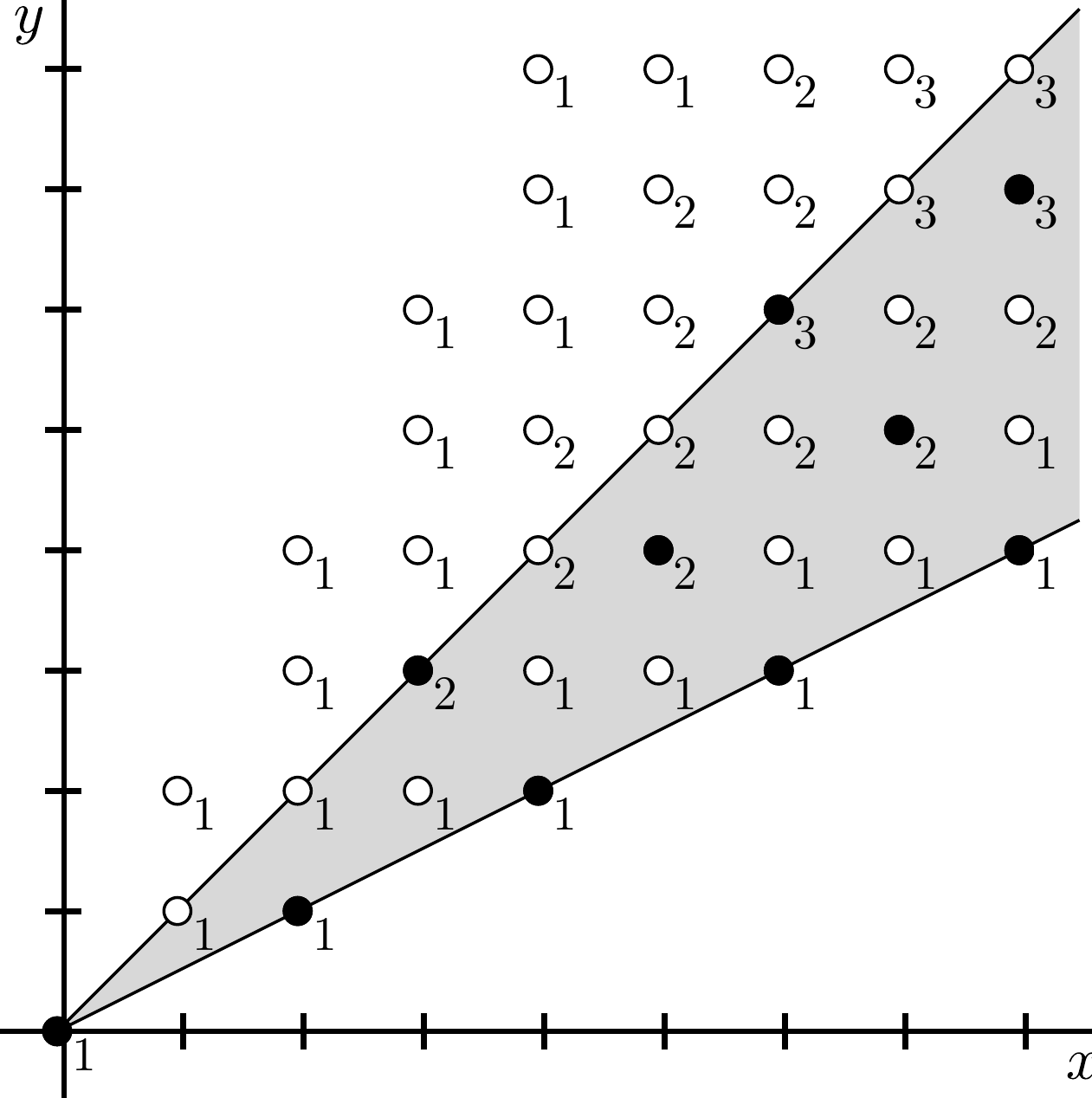}
&
\includegraphics[width=1.5in]{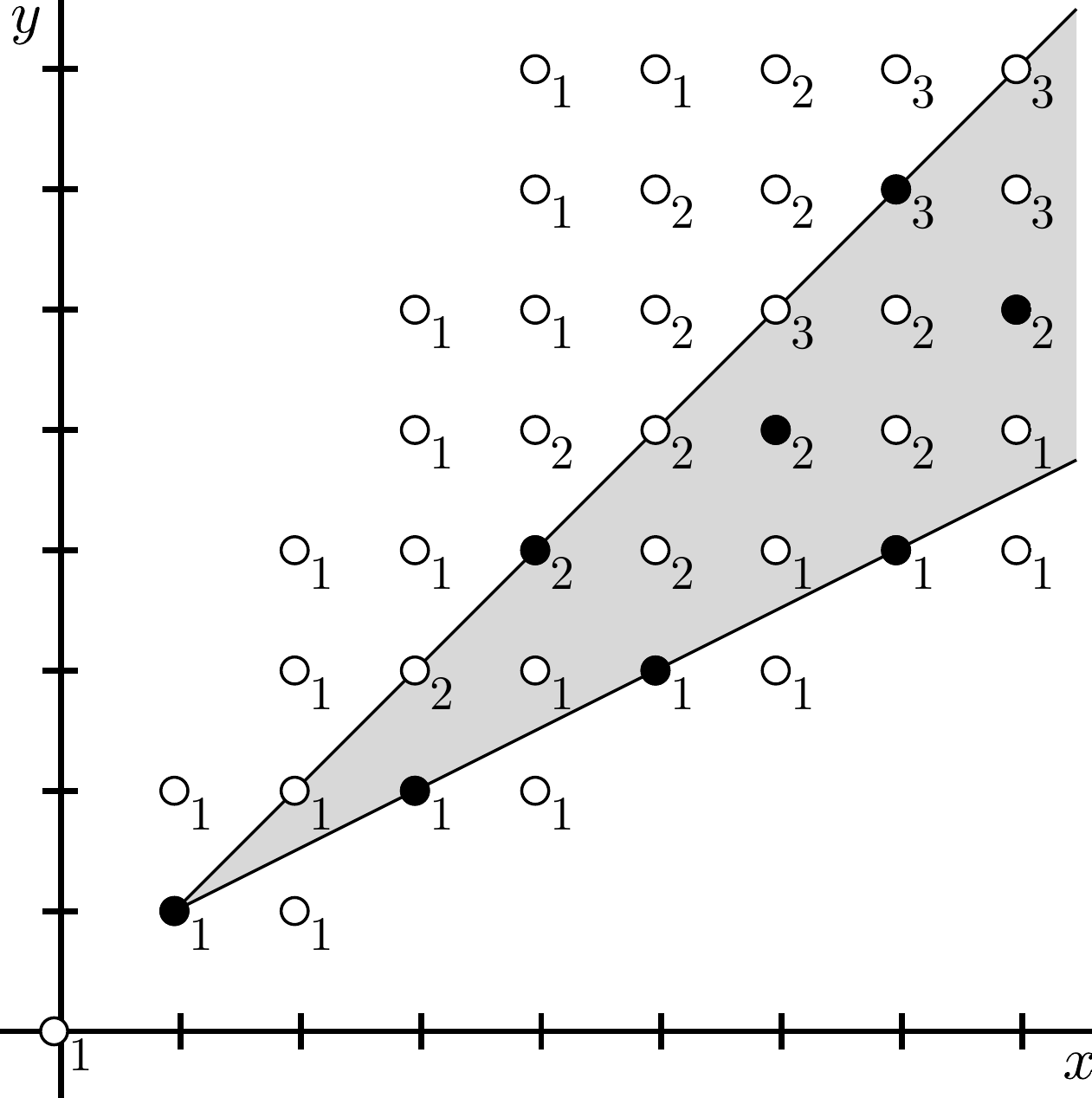}
&
\includegraphics[width=1.5in]{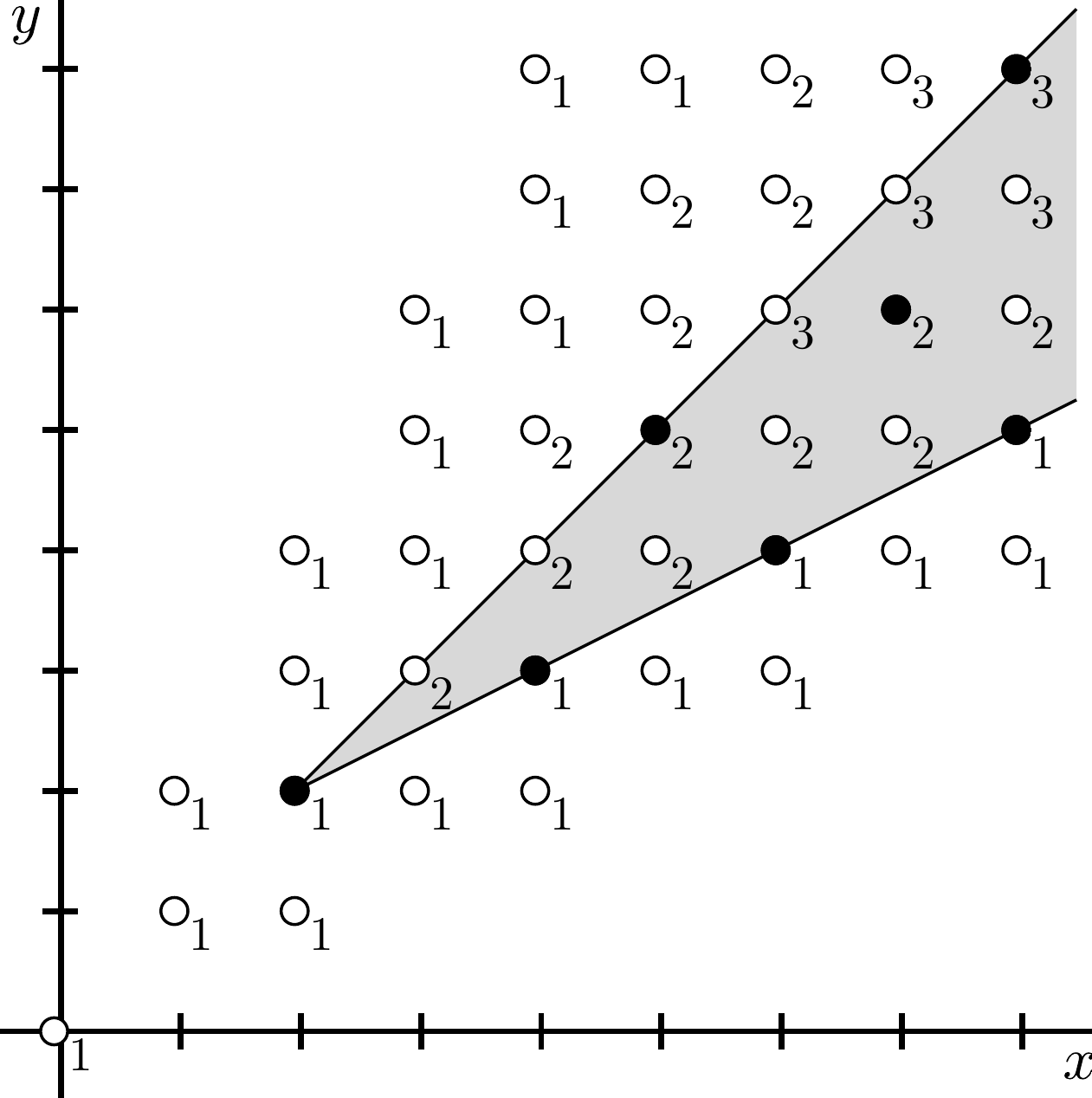}
\\
$\scriptstyle{C((0,0);(2,1),(3,3))}$
&
$\scriptstyle{C((1,1);(2,1),(3,3))}$
&
$\scriptstyle{C((2,2);(2,1),(3,3))}$
\end{tabular}
\caption{The values above represent the number of distinct factorizations of elements of $\Gamma = \<(1,2),(1,1),(2,1)\> \subset \NN^2$.  The filled dots in each plot depict one of the cones in Example~\ref{e:factset}.}  
\label{f:factset}
\end{figure}

\begin{defn}\label{d:factring}
Suppose $\Gamma = \<\alpha_1, \ldots, \alpha_r\> \subset A$.  
The $A$-graded ring $R_\Gamma = \kk[y_1, \ldots, y_r]$ with $\deg(y_i) = \alpha_i$ for each $i$ is called the \emph{ring of factorizations of $\Gamma$}.  
\end{defn}

Proposition~\ref{p:factorhilbert} justifies the nomenclature in Definition~\ref{d:factring}.  

\begin{prop}\label{p:factorhilbert}
Suppose $\Gamma = \<\alpha_1, \ldots, \alpha_r\> \subset A$.  
The equality
$$\mathcal H(R_\Gamma;\alpha) = |\mathsf Z_\Gamma(\alpha)|$$
holds for all $\alpha \in A$.  In particular, the function $\Gamma \to \NN$ given by $\alpha \mapsto |\mathsf Z_\Gamma(\alpha)|$ is eventually quasipolynomial of cumulative degree $r$.  
\end{prop}

\begin{proof}
Each monomial $\yy^\aa = y_1^{a_1} \cdots y_k^{a_r} \in R_\Gamma$ has degree $\alpha = a_1\alpha_1 + \cdots + a_r\alpha_r \in \Gamma$.  This gives, for each $\alpha \in A$, a bijection between the set $\mathsf Z_\Gamma(\alpha)$ of factorizations of $\alpha$ in $\Gamma$ and the set of degree $\alpha$ monomial elements of $R_\Gamma$.  In particular, this means $\mathcal H(R_\Gamma;\alpha) = |\mathsf Z_\Gamma(\alpha)|$, and the second claim follows by Theorem~\ref{t:hilbert2}.  
\end{proof}

\begin{remark}\label{r:factset}
In view of Proposition~\ref{p:factorhilbert}, the eventual quasipolynomial given in Example~\ref{e:factset} for the number of factorizations of $\Gamma = \<(2,1),(1,1),(1,2)\> \subset \NN^2$ can be verified (and in fact, derived) by examining the generating function of $\mathcal H(R_\Gamma;-)$, called the \emph{Hilbert series} of $R_\Gamma$.  See~\cite{multiquasi} for more detail on such computations.  
\end{remark}

Theorem~\ref{t:factorasymp} is a consequence of the bijection established in Proposition~\ref{p:factorhilbert} that strengthens \cite[Theorem~1.1]{numfactoratomic} and \cite[Theorem~1]{factorasymp} for finitely generated semigroups.  

\begin{thm}\label{t:factorasymp}
Fix $\alpha \in \Gamma = \<\alpha_1, \ldots, \alpha_r\> \subset A$.  
Let $r(\alpha)$ denote the maximal number of linearly independent factorizations of multiples of $\alpha$ in $\Gamma$, that is, 
$$r(\alpha) = \dim_\QQ\spann_\QQ(\textstyle\bigcup_{k \ge 0} \mathsf Z_\Gamma(k\alpha)).$$
The function $|\mathsf Z_\Gamma(k\alpha)|$ is eventually quasipolynomial in $k$ of degree $r(\alpha) - 1$ whose leading coefficient is constant.  In particular, for some $B(\alpha) \in \QQ_{>0}$, we have
$$|\mathsf Z_\Gamma(k\alpha)| = B(\alpha)k^{r(\alpha)-1} + O(k^{r(\alpha)-2})$$
for $k$ sufficiently large.  
\end{thm}

\begin{proof}
By Proposition~\ref{p:factorhilbert} and Theorem~\ref{t:affineequiv}(c), $|\mathsf Z_\Gamma(k\alpha)|$ is eventually quasipolynomial in $k$ of degree at most $r$.  Let $f$ denote this quasipolynomial, and consider the subring 
$$R = \kk[\yy^\aa : \aa \in \mathsf Z_\Gamma(k\alpha), k \ge 0] \subset R_\Gamma$$
whose monomials correspond to the factorizations of $k\alpha$ for some $k \ge 0$.  Each monomial in $R$ has degree $k\alpha$ for some $k \ge 0$, so $R$ can be $\NN$-graded with $\deg(\yy^\aa) = k$ for $\aa \in \mathsf Z_\Gamma(k\alpha)$.  This implies $\mathcal H(R;k) = f(k)$ for $k \gg 0$.  Since $\dim R = r(\alpha)$, we have $\deg(f) = r(\alpha) - 1$.  Additionally, $R$ has at least one generator of degree 1 since $\mathsf Z_\Gamma(\alpha) \ne \emptyset$, so the ideal $I$ defined in Theorem~\ref{t:constcoeffs} is nonempty.  This ensures the leading term of $f$ is constant.  
\end{proof}

Theorem~\ref{t:factornumerical} specializes Theorem~\ref{t:factorasymp} to numerical semigroups $\Gamma$, resulting in a closed form for the constant leading coefficient of $|\mathsf Z_\Gamma(-)|$ in this setting (Corollary~\ref{c:factornumerical}).  

\begin{thm}\label{t:factornumerical}
Fix a numerical semigroup $\Gamma = \<n_1, \ldots, n_r\> \subset \NN$.  There exist periodic functions $a_0, \ldots, a_{d-2}:\NN \to \QQ$, each with period dividing $\lcm(n_1, \ldots, n_r)$, such that
$$|\mathsf Z_\Gamma(n)| = \frac{1}{(r-1)! n_1 \cdots n_r}n^{r-1} + a_{r-2}(n)n^{r-2} + \cdots + a_1(n)n + a_0(n)$$
for all $n \ge 0$.
\end{thm}

\begin{proof}
Since $\dim R_\Gamma = r$, Proposition~\ref{p:factorhilbert} and Theorem~\ref{t:hilbert} imply $|\mathsf Z_\Gamma(n)| = f(n)$ for $n \gg 0$, where $f:\NN \to \NN$ is a quasipolynomial of degree $r-1$ with period dividing $\lcm(n_1, \ldots, n_r)$.  Let $a_0, \ldots, a_{r-1}:\NN \to \QQ$ denote periodic functions such that 
$$f(n) = a_{r-1}(n)n^{r-1} + \cdots + a_1(n)n + a_0(n)$$
for all $n \in \NN$.  

To prove that $|\mathsf Z_\Gamma(n)| = f(n)$ for all $n \ge 0$, we proceed by induction on $r$.  If $r = 1$, then $R_\Gamma = \kk[y_1]$, so $\mathcal H(R_\Gamma;n) = 1$ for all $n \ge 0$, which is clearly a quasipolynomial of the desired form.  
Now, suppose $r \ge 2$, let $c = \gcd(n_1, \ldots, n_{r-1})$, and let  $\Gamma' = \<n_1/c, \ldots, n_{r-1}/c\> \subset \Gamma$.  By induction, $\mathcal H(R_{\Gamma'};n)$ equals a quasipolynomial 
$$g(n) = \frac{c^{r-1}}{(r-2)! n_1 \cdots n_{r-1}}n^{r-2} + b_{r-3}(n)n^{r-3} + \cdots + b_1(n)n + b_0(n)$$
with period dividing $\lcm(n_1/c, \ldots, n_{r-1}/c)$, for all $n \ge 0$.  The sequence 
$$0 \longrightarrow R_\Gamma(-n_r) \xrightarrow{\cdot y_r} R_\Gamma \longrightarrow R_\Gamma/\<y_r\> \longrightarrow 0$$
is exact, and yields the equality 
$$\mathcal H(R_\Gamma;n) - \mathcal H(R_\Gamma;n - n_r) = \mathcal H(R_\Gamma/\<y_r\>;n) = 
\left\{\begin{array}{ll}
g(n/c) & c \mid n \\
0 & c \nmid n
\end{array}\right.$$
on Hilbert functions.  Let $G(n)$ denote the function on the right hand side in the above equality.  This means $f(n) - f(n - n_r) = G(n)$ for $n \gg 0$, but since $f$ is determined by finitely many values, this equality must hold for all $n \ge 0$.  Furthermore, $\mathcal H(R_\Gamma;n) = f(n)$ for all $n \ge 0$ since $\mathcal H(R_\Gamma;n) - \mathcal H(R_\Gamma;n - n_r) = G(n)$.  

Now, it remains to show that $a_{r-1}(n)$ has the desired form.  Since $G$ has degree strictly less than $r-1$, comparing coefficients yields the equalities $a_{r-1}(n) = a_{r-1}(n - n_r)$ and 
$$a_{r-2}(n) - (a_{r-2}(n - n_r) - (r-1)n_ra_{r-1}(n - n_r)) = 
\left\{\begin{array}{ll}
\frac{c}{(r-2)! n_1 \cdots n_{r-1}} & c \mid n \\
0 & c \nmid n
\end{array}\right.$$
for all $n$.  Let $\pi = \lcm(n_1, \ldots, n_r)$.  Since $\gcd(c,n_r) = 1$, we have 
$$\begin{array}{rcl}
\displaystyle\frac{\pi}{(r-2)! n_1 \cdots n_r} &=& \displaystyle\sum_{i = 1}^{\pi/n_r} a_{r-2}(n - (i-1)n_r) - (a_{r-2}(n - in_r) - (r-1)n_ra_{r-1}(n - in_r)) \\
&=& a_{r-2}(n) - a_{r-2}(n - \pi) + (r-1)\pi a_{r-1}(n),
\end{array}$$
and since $a_{r-2}$ is $\pi$-periodic, this yields the desired equality.  
\end{proof}

\begin{cor}\label{c:factornumerical}
Fix a numerical semigroup $\Gamma = \<n_1, \ldots, n_r\> \subset \NN$ and an element $n \in \Gamma$.  Resuming the notation from Theorem~\ref{t:factorasymp}, we have $r(n) = r$ and 
$$B(n) = n^{r-1}/(r-1)! n_1 \cdots n_r.$$
\end{cor}

\begin{example}\label{e:factornumerical}
Consider the numerical semigroup $\Gamma = \<6,9,20\> \subset \NN$.  By Theorem~\ref{t:factornumerical}, there exist periodic functions $a_0, a_1:\NN \to \QQ$ such that
$$|\mathsf Z_\Gamma(n)| = \frac{1}{2160}n^2 + a_1(n)n + a_0(n)$$
for all $n \in \Gamma$.  Computing $|\mathsf Z_\Gamma(n)|$ for all $n \le 2 \cdot \lcm(6,9,20) = 360$ in \texttt{Sage} \cite{sage} shows the linear coefficient $a_1$ has period 6 and the constant coefficient $a_0$ has full period $180$.  
\end{example}

\begin{remark}\label{r:ehrhart}
The Hilbert function in Proposition~\ref{p:factorhilbert} is the only one constructed in this paper that is quasipolynomial for all $\alpha \in \Gamma$.  Algebraically, this is because the start of quasipolynomial behavior of a Hilbert function is controlled by the algebraic relations (and higher syzygies) of the underlying module, and the polynomial ring has no relations between its generators.  On the other hand, each graded module $M$ constructed throughout the rest of the paper has some nontrivial algebraic relations (or are defined over a $\kk$-algebra with nontrivial relations).  For an interesting development on which general conditions enable a function to be eventually quasipolynomial, see~\cite{presburgerarith}.  

The absense of an ``$n \gg 0$'' assumption in Theorem~\ref{t:factornumerical} can also be interpreted geometrically.  In particular, when $\Gamma$ is a numerical semigroup, the function $\mathsf Z_\Gamma$ coincides with the Ehrhart function of a rational simplex, which is quasipolynomial by Ehrhart's theorem \cite{continuousdiscretely}.  The algebraic relations found in many of the modules constructed later in this paper can be viewed as inducing an equivalence relation on the lattice points in dilations of this simplex, and the corresponding Hilbert function counts equivalence classes.  The interested reader is encouraged to consult \cite[Chapter~12]{cca} for details on the connection between Hilbert functions and Ehrhart functions.  
\end{remark}

\section{The delta set}\label{s:delta}

In this section, we consider the delta set invariant (Definition~\ref{d:deltaset}), which measures the ``gaps'' in a semigroup element's factorization lengths.  The main result is Theorem~\ref{t:deltahilbert}, in which we construct an ascending chain of ideals in the ring $R_\Gamma$ of factorizations of a semigroup $\Gamma \subset A$ (Definition~\ref{d:factring}) such that the Hilbert functions of successive quotients in this chain determine the delta sets of the elements of $\Gamma$.  Applying Theorem~\ref{t:hilbert2} to Theorem~\ref{t:deltahilbert} yields a classification of the delta set for all such semigroups $\Gamma$ (Corollary~\ref{c:deltahilbert}).  Furthermore, applying Theorem~\ref{t:hilbert} to the special case of Theorem~\ref{t:deltahilbert} where $\Gamma$ is a numerical semigroup yields Corollary~\ref{c:deltanumerical}, a recent result appearing as \cite[Corollary~18]{compasympdelta} as an improvement on \cite[Theorem~1]{deltaperiodic}.   Theorem~\ref{t:deltahilbert} also has computational applications; see Remark~\ref{r:deltaalgorithm}.  


\begin{defn}\label{d:deltaset}
Fix $\alpha \in \Gamma = \<\alpha_1, \ldots, \alpha_r\> \subset A$.  
Given $\aa \in \mathsf Z_\Gamma(\alpha)$, the \emph{length of $\aa$} is the number $|\aa| = a_1 + \cdots + a_r$ of irreducibles in $\aa$.  The \emph{length set of $\alpha$} is the set
$$\mathsf L_\Gamma(\alpha) = \{a_1 + a_2 + \cdots + a_r : \aa \in \mathsf Z_\Gamma(\alpha)\}$$
of factorization lengths.  Writing $\mathsf L_\Gamma(\alpha) = \{\ell_1 < \cdots < \ell_m\}$, the \emph{delta set of $\alpha$} is the set
$$\Delta(\alpha) = \{\ell_{i+1} - \ell_i : 1 \le i < m\}$$
of successive differences of factorization lengths.  The \emph{delta set of $\Gamma$} is $\Delta(\Gamma) = \bigcup_{\alpha \in \Gamma} \Delta(\alpha)$.  
We say $\Gamma$ is \emph{half-factorial} if $|\mathsf L_\Gamma(\alpha)| = 1$ for all $\alpha \in \Gamma$.  
\end{defn}

\begin{defn}\label{d:lenideal}
Suppose $\Gamma = \<\alpha_1, \ldots, \alpha_r\> \subset A$.  
The \emph{length set ideal of $\Gamma$} is 
$$I_\mathsf L = \<\yy^\aa - \yy^\bb : \aa, \bb \in \mathsf Z_\Gamma(\alpha) \text{ for some } \alpha \in \Gamma \text{ and } |\aa| = |\bb|\> \subset R_\Gamma,$$
a homogeneous ideal in the ring of factorizations $R_\Gamma$ of $\Gamma$.  
\end{defn}

\begin{remark}\label{r:halffactorial}
The ``half-factorial'' assumption in Proposition~\ref{p:lenhilbert} and Theorem~\ref{t:lenasymp} is necessary, as otherwise $|\mathsf L_\Gamma(\alpha)| = 1$ for all nonzero $\alpha \in \Gamma$, which is (quasi)constant.  
\end{remark}

\begin{prop}\label{p:lenhilbert}
Suppose $\Gamma = \<\alpha_1, \ldots, \alpha_r\> \subset A$.  The equality 
$$\mathcal H(R_\Gamma/I_\mathsf L;\alpha) = |\mathsf L_\Gamma(\alpha)|$$
holds for all $\alpha \in \Gamma$.  In particular, the function $\Gamma \to \NN$ given by $\alpha \mapsto |\mathsf L_\Gamma(\alpha)|$ is eventually quasilinear if $\Gamma$ is not half-factorial.  
\end{prop}

\begin{proof}
By Proposition~\ref{p:factorhilbert}, the monomials $\yy^\aa$ of $R_\Gamma$ of degree $\alpha$ are in bijection with the factorizations of $\alpha$.  The quotient by $I_\mathsf L$ is graded since $I_\mathsf L$ is homogeneous, and two monomials $\yy^\aa$ and $\yy^\bb$ of the same degree have the same image modulo $I_\mathsf L$ precisely when their factorization lengths coincide.  Thus, modulo $I_\mathsf L$, the monomials of degree~$\alpha$ are in bijection with the set $\mathsf L(\alpha)$, so $\mathcal H(R_\Gamma/I_\mathsf L;\alpha) = |\mathsf L_\Gamma(\alpha)|$, which by Theorem~\ref{t:hilbert2} is eventually quasipolynomial of cumulative degree $\dim R_\Gamma/I_\mathsf L$.  

It remains to show that $|\mathsf L_\Gamma(-)|$ is eventually quasilinear when $\Gamma$ is not half-factorial.  
First, assume $\Gamma \subset \NN^d$ is affine.  The ideal $I_\mathsf L$ is the kernel of the monomial map $\kk[y_1, \ldots, y_r] \to \kk[x_1, \ldots, x_d, z]$ sending 
$y_i \mapsto \xx^{\alpha_i}z$, 
since two monomials $\yy^\aa$ and $\yy^\bb$ have the same image precisely when $\aa$ and $\bb$ are equal-length factorizations of the same element of $\Gamma$.  This means 
$$\dim R_\Gamma/I_\mathsf L = \dim\spann_\QQ\{(\alpha_i,1) \in \NN^{d+1} : 1 \le i \le r\},$$
which can only be $\dim\spann_\QQ(\Gamma)$ or $\dim\spann_\QQ(\Gamma) + 1$ since projecting along the last coordinate yields $\spann_\QQ(\Gamma) = \dim R_\Gamma$.  By assumption, $\Gamma$ is not half-factorial, so this projection is not injective, and $\dim R_\Gamma/I_\mathsf L = \dim\spann_\QQ(\Gamma) + 1$.  It follows that $|\mathsf L_\Gamma(-)|$ has eventual degree 1.  

Lastly, suppose $\Gamma \subset A$ is not necessarily affine.  The image $\rho(\Gamma)$ under the projection map $\rho:A \to \NN^d$ is affine, and the image of any factorization of $\alpha \in \Gamma$ is a factorization for $\rho(\alpha) \in \rho(\Gamma)$.  As such, $|\mathsf L_\Gamma(\alpha)| \le |\mathsf L_{\rho(\Gamma)}(\rho(\alpha))|$ for all $\alpha \in \Gamma$, so by the above argument $|\mathsf L_\Gamma(-)|$ has eventual degree at most 1.  Again, $\Gamma$ is not half-factorial, so equality must hold.  
\end{proof}

As a consequence of the bijection established in Proposition~\ref{p:lenhilbert}, we obtain Theorem~\ref{t:lenasymp}, an asymptotic characterization of the cardinality of semigroup element length sets, which also follows as a consequence of \cite[Theorem~4.9.2]{nonuniq}.  

\begin{thm}\label{t:lenasymp}
Suppose $\Gamma = \<\alpha_1, \ldots, \alpha_r\> \subset A$ is not half-factorial, and fix $\alpha \in \Gamma$.  
There exists a positive constant $B(\alpha)$ and a periodic function $a_0$ such that 
$$|\mathsf L_\Gamma(k\alpha)| = B(\alpha)k + a_0(k)$$
for $k \gg 0$.  
\end{thm}

\begin{proof}
As in the proof of Theorem~\ref{t:factorasymp}, consider the subring 
$$R = \kk[\yy^\aa : \aa \in \mathsf Z_\Gamma(k\alpha), k \ge 0] \subset R_\Gamma$$
whose monomials correspond to factorizations of $k\alpha$ for some $k \ge 0$ under the bijection established in Proposition~\ref{p:factorhilbert}, and whose grading is given by $\deg(\yy^\aa) = k$ for each $\aa \in \mathsf Z(k\alpha)$.  Letting $I = I_\mathsf L \cap R$, Proposition~\ref{p:lenhilbert} ensures that $|\mathsf L(k\alpha)| = \mathcal H(R/I;k)$ is eventually quasilinear in $k$.  Moreover, $R/I$ has at least one monomial of degree~1 since $\mathsf Z(\alpha)$ is nonempty, so Theorem~\ref{t:constcoeffs} ensures the existence of $B(\alpha)$.  
\end{proof}

Theorem~\ref{t:lennumerical} is the special case of Proposition~\ref{p:lenhilbert} for numerical semigroups.  

\begin{thm}\label{t:lennumerical}
Fix a numerical semigroup $\Gamma = \<n_1, \ldots, n_r\>$.  There exists a periodic function $a_0:\NN \to \QQ$ whose period divides $\lcm(n_1,n_r)$ and a constant $a_1$ such that 
$$|\mathsf L_\Gamma(n)| = a_1n + a_0(n)$$
for $n \gg 0$.  
\end{thm}

\begin{proof}
Applying Proposition~\ref{p:lenhilbert} and Theorem~\ref{t:hilbert} proves $|\mathsf L_\Gamma(n)|$ is eventually quasilinear.  Fix periodic functions $a_0, a_1:\NN \to \QQ$ such that 
$$|\mathsf L_\Gamma(n)| = a_1(n)n + a_0(n)$$
for $n \gg 0$, let $f(n) = a_1(n)n + a_0(n)$, and let $\pi$ denote the period of $f$.  

First, we claim $(y_1, y_r)$ is a homogeneous system of parameters for $M = R_\Gamma/I_\mathsf L$, from which we conclude $\pi \mid \lcm(n_1, n_r)$ by Theorem~\ref{t:hilbert}.  Indeed, since $\Gamma$ is cancellative, $y_1$ is a nonzerodivisor on $M$.  Moreover, for any $k \ge 0$, $y_r^k \in M$ has nonzero image modulo $y_1M$ since $k\mathbf e_r \in \mathsf Z_\Gamma(kn_r)$ is the unique factorization of $kn_r$ of length $k$.  Observing that some power of each $y_i$ has zero image in $M/\<y_1, y_r\>M$ proves the claim.  

It remains to prove that $a_1$ is constant.  If $\gcd(n_1,n_r) > 1$, then some $y_i$ has degree relatively prime to $\lcm(n_1,n_r)$.  On the other hand, if $\gcd(n_1, n_r) = 1$, then $y_1y_r$ has degree $n_1 + n_r$, and $\gcd(n_1 + n_r, n_1n_r) = \gcd(n_1,(n_1-1)n_r) = 1$.  In either case, Theorem~\ref{t:constcoeffs} completes the proof.  
\end{proof}

\begin{remark}\label{r:lenleadingcoeff}
An explicit formula for the leading coefficient of the quasilinear function in Theorem~\ref{t:lennumerical} is given in Corollary~\ref{c:lenleadingcoeff}, as the proof relies on several upcoming results.  If we wanted, we could appeal to existing results on length sets (see, for instance, \cite[Chapter~4]{nonuniq}), but our chosen proof demonstrates how the algebro-combinatorial framework presented in this paper can be used to discern many of these same results.  
\end{remark}

\begin{example}\label{e:lennumerical}
Let $\Gamma = \<6,9,20\> \subset \NN$.  The length set ideal of $\Gamma$ is given by 
$$I_\mathsf L = \<x^{11}z^3 - y^{14}\> \subset R_\Gamma = \kk[x,y,z]$$
where $\deg(x) = 6$, $\deg(y) = 9$ and $\deg(z) = 20$.  The degree of both monomials in the generator of $I_\mathsf L$ is 126, which is the smallest element of $\Gamma$ with two distinct factorizations of equal length.  Moreover, there exists a function $a_0:\NN \to \QQ$ with period 60 such that
$$\mathsf L(n) = \textstyle\frac{7}{60}n + a_0(n)$$
for all $n \ge 92$.  Note that this bound is sharp, as the quasilinear function above does not coincide with $\mathsf L(n)$ for $n = 91$; this can be verified by a simple computation.  
\end{example}



We are now ready to state and prove Theorem~\ref{t:deltahilbert}, which implies that the set of elements of a semigroup $\Gamma \subset A$ having a given value in their delta set equals the support of an eventually quasipolynomial function.  Applying Theorem~\ref{t:affineequiv} immediately yields Corollary~\ref{c:deltahilbert}, which gives a more explicit description of this set.  

\begin{thm}\label{t:deltahilbert}
Suppose $\Gamma = \<\alpha_1, \ldots, \alpha_k\> \subset A$.  
The ideals
$$I_j = \<\yy^\aa - \yy^\bb : \aa, \bb \in \mathsf Z_\Gamma(\alpha), \alpha \in \Gamma, \text{and } \big||\bb| - |\aa|\big| \le j\> \subset R_\Gamma$$
for $j \ge 0$ form an ascending chain 
$$I_\mathsf L = I_0 \subset I_1 \subset I_2 \subset \cdots$$
in which $\mathcal H(I_j/I_{j-1};\alpha)$ counts the number of successive length differences in $\mathsf L(\alpha)$ equal to~$j$ whenever $j \ge 1$.  In particular, $I_{j-1} \subsetneq I_{j}$ if and only of $j \in \Delta(\Gamma)$.  
\end{thm}

\begin{proof}
It is immediate from the definition that $I_{j-1} \subset I_j$ for all $j \ge 1$.  Fix $\alpha \in \Gamma$ and factorizations $\aa, \bb \in \mathsf Z(\alpha)$ with $|\bb| - |\aa| = j \ge 1$.  The binomial $\yy^\aa - \yy^\bb \in I_j$ lies in $I_{j-1}$ precisely when there is a factorization $\cc \in \mathsf Z(\alpha)$ such that $|\aa| < |\cc| < |\bb|$, since $\yy^\aa - \yy^\bb = (\yy^\aa - \yy^\cc) + (\yy^\cc - \yy^\bb)$.  It follows that (i) $I_{j-1} \subsetneq I_j$ if and only if $j \in \Delta(\Gamma)$, and (ii) $\mathcal H(I_j/I_{j-1};\alpha)$ yields the desired quantity.  
\end{proof}

\begin{cor}\label{c:deltahilbert}
Suppose $\Gamma = \<\alpha_1, \ldots, \alpha_k\> \subset A$.  
For each $j \in \Delta(\Gamma)$, the set 
$$\{\alpha \in \Gamma : j \in \Delta(\alpha)\} \subset \Gamma$$
is a disjoint union of finitely many cones.  
\end{cor}

\begin{proof}
This follows from Theorems~\ref{t:affineequiv}(c),~\ref{t:hilbert2} and~\ref{t:deltahilbert}.  
\end{proof}

Specializing Theorem~\ref{t:deltahilbert} to numerical semigroups yields Corollary~\ref{c:deltanumerical}.  

\begin{cor}[{\cite[Theorem~1]{deltaperiodic}}]\label{c:deltanumerical}
For any numerical semigroup $\Gamma = \<n_1, \ldots, n_r\> \subset \NN$, the function $\Delta:\Gamma \to 2^\NN$ is eventually periodic with period dividing $\lcm(n_1,n_r)$.  
\end{cor}

\begin{proof}
Applying Theorems~\ref{t:hilbert} and~\ref{t:deltahilbert} proves $\Delta:\Gamma \to 2^\NN$ is eventually periodic, and Theorem~\ref{t:lennumerical} produces the desired bound on its period.  
\end{proof}

\begin{remark}\label{r:deltaacc}
It is known that $\Delta(\Gamma)$ is finite for any finitely generated semigroup $\Gamma$ (see, for instance, \cite[Theorem~3.14]{nonuniq}).  We also recover this fact as a consequence of Theorem~\ref{t:deltahilbert} and the ascending chain condition on $R_\Gamma$.  
\end{remark}

The following examples use \texttt{Sage}~\cite{sage} and the \texttt{GAP} package \texttt{numericalsgps} \cite{numericalsgpsgap}.  

\begin{example}\label{e:deltaaffine}
Let $\Gamma = \<(1,1),(1,5),(2,5),(3,5),(5,1),(5,2),(5,3)\> \subset \NN^2$.  The delta set of $\Gamma$ is $\Delta(\Gamma) = \{1,2,4\}$, and Figure~\ref{f:deltaaffine} depicts which elements of $\Gamma$ have each of these values in their delta set.  Using notation from Theorem~\ref{t:deltahilbert}, $I_2 = I_3$ since $3 \notin \Delta(\Gamma)$.  
\end{example}

\begin{figure}
\begin{tabular}{c@{\hspace{0.2in}}c@{\hspace{0.2in}}c}

\includegraphics[width=1.7in]{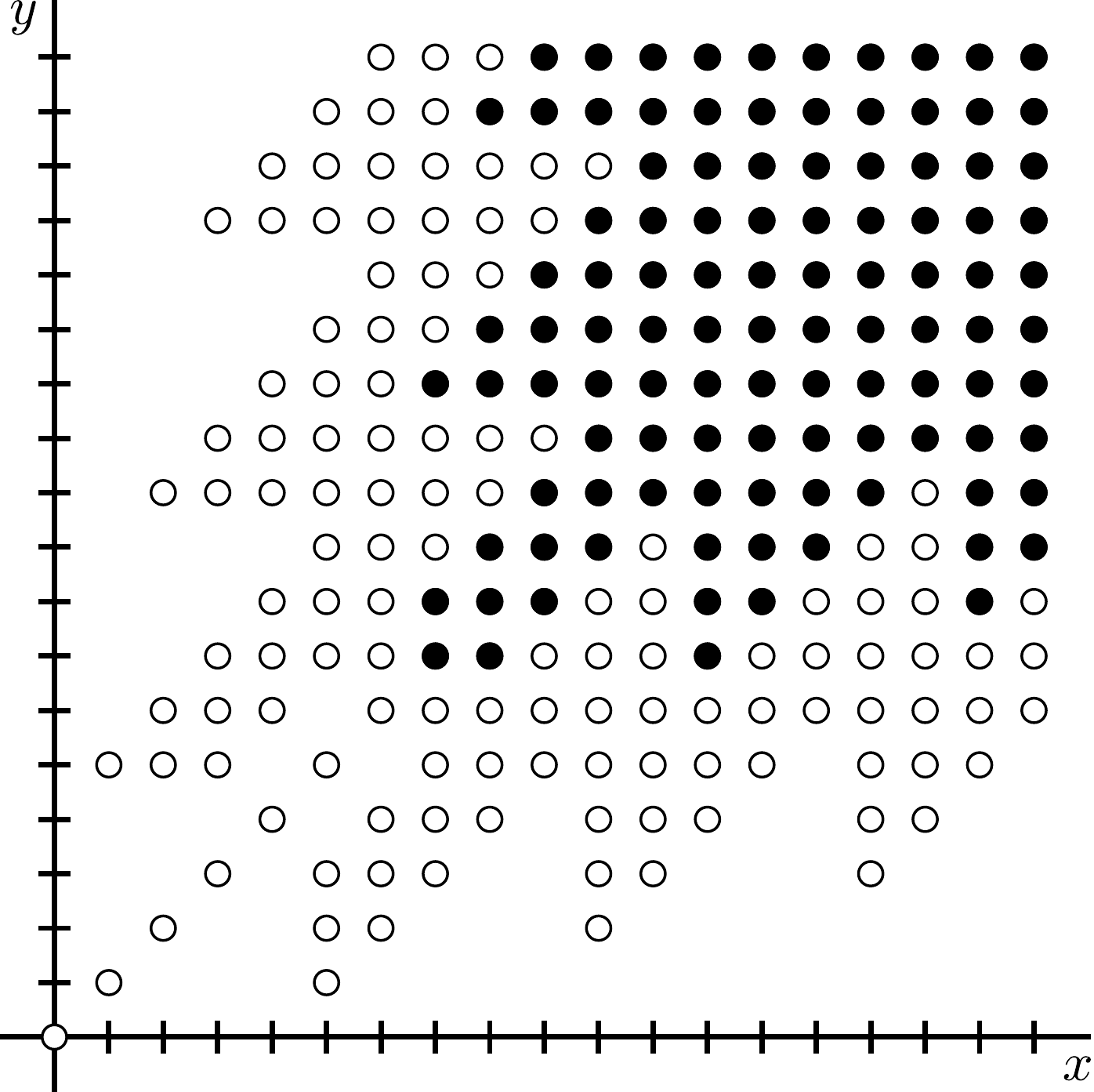}
&
\includegraphics[width=1.7in]{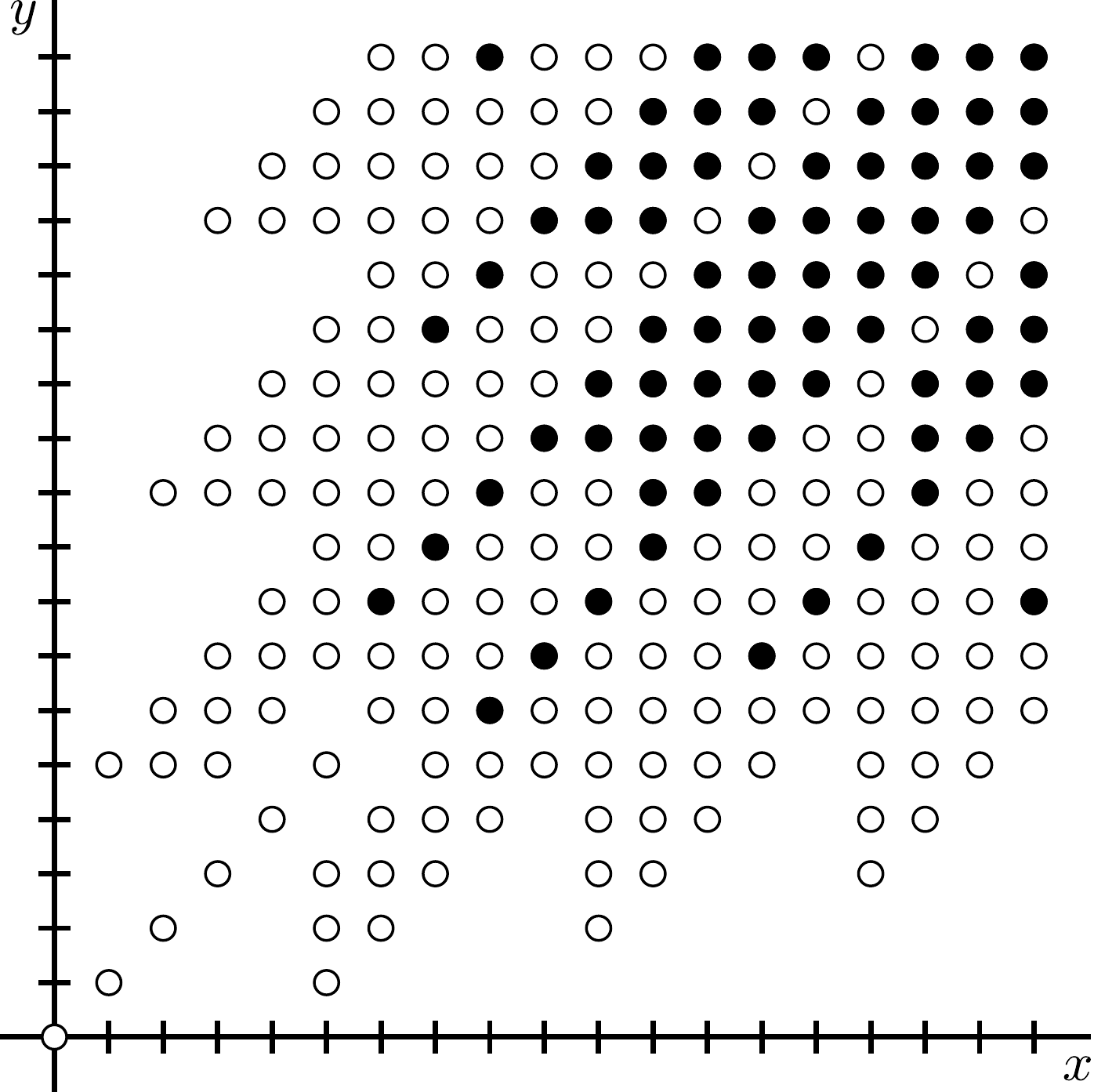}
&
\includegraphics[width=1.7in]{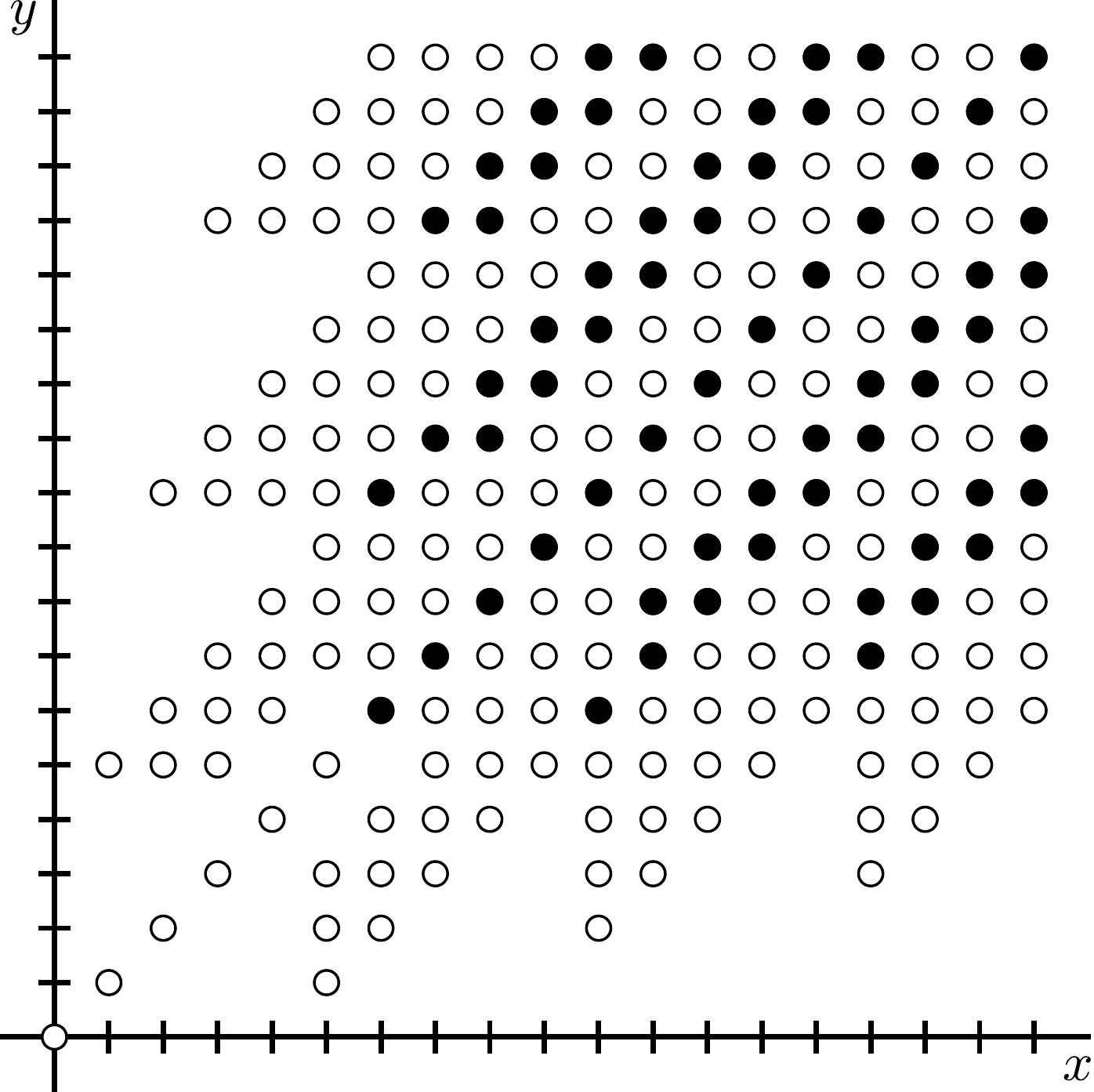}
\\
$\scriptstyle{\{\alpha \in \Gamma~:~1 \in \Delta(\alpha)\}}$
&
$\scriptstyle{\{\alpha \in \Gamma~:~2 \in \Delta(\alpha)\}}$
&
$\scriptstyle{\{\alpha \in \Gamma~:~4 \in \Delta(\alpha)\}}$
\end{tabular}
\caption{For $\Gamma \subset \NN^2$ as in Example~\ref{e:deltaaffine}, each filled dot denotes an element of $\Gamma$ with the specified value in its delta set.}  
\label{f:deltaaffine}
\end{figure}

\begin{example}\label{e:deltanumerical}
Let $\Gamma = \<6,9,20\> \subset \NN$ denote the numerical semigroup from Example~\ref{e:lennumerical}.  Resuming notation from Theorem~\ref{t:deltahilbert}, we have 
$$\begin{array}{rcl}
I_\mathsf L = I_0 = \<x^{11}z^3 - y^{14}\> 
&\subsetneq& I_1 = I_0 + \<x^3 - y^2, x^8z^3 - y^{12}\> \\
&\subsetneq& I_2 = I_1 + \<x^5z^3 - y^{10}\> \\
&\subsetneq& I_3 = I_2 + \<x^2z^3 - y^8\> \\
&\subsetneq& I_4 = I_3 + \<xy^6 - z^3\> = I_5 = I_6 = \cdots
\end{array}$$
The quotient $I_1/I_0$ has dimension 1, and $\mathcal H(I_1/I_0;n) > 0$ for all $n \ge 62$, meaning $\{n \in \Gamma : 1 \in \Delta(n)\}$ has eventual period 1.  The remaining nonzero quotients are each dimension 0, and the sets $\{n \in \Gamma : j \in \Delta(n)\}$ for $j = 2,3,4$ have period $20$ for $n \ge 92$, $n \ge 74$, and $n \ge 56$, respectively (based on computation, each of these bounds is sharp as well).  Figure~\ref{f:deltanumerical} depicts these sets, demonstrating that $\Delta:\Gamma \to 2^\NN$ is periodic for $n \ge \max(62,92,74,56) = 92$ with period $\lcm(1,20) =~20$.  Notice that $x^6 - y^4 \in I_2$ satisfies $\deg(x^6 - y^4) = 36 < \deg(x^5z^3 - y^{10}) = 90$, but since we can write $x^6 - y^4 = (x^6 - x^3y^2) + (x^3y^2 - y^4) \in I_1$, it does not constitute a generator of $I_2/I_1$.  
\end{example}

\begin{figure}
\includegraphics[width=5.6in]{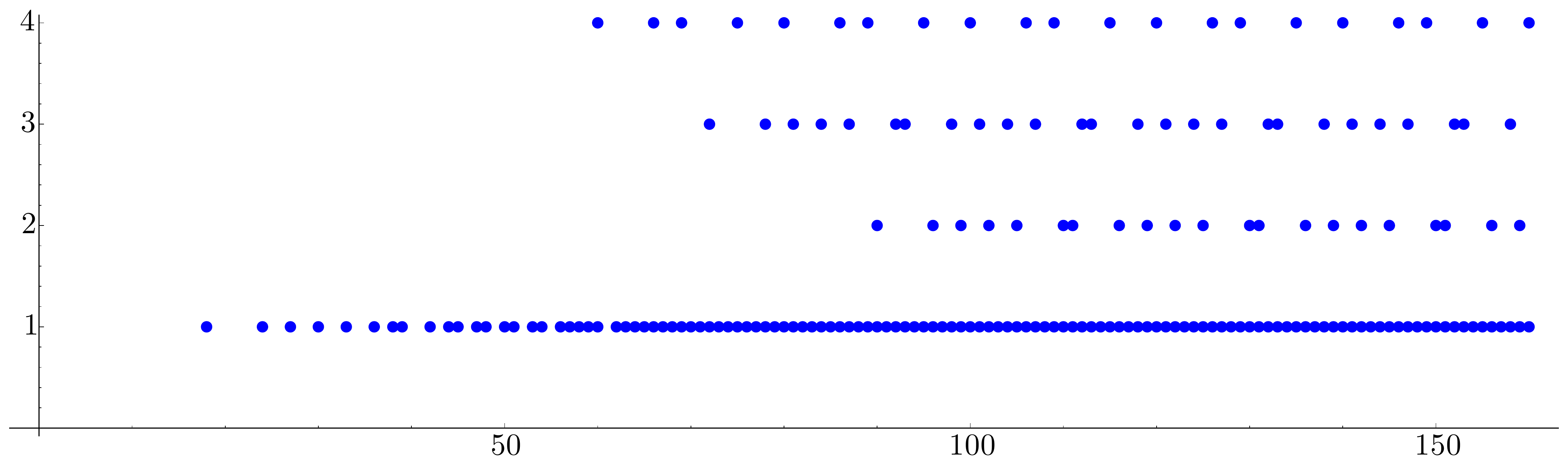}
\medskip
\caption{A plot showing the delta sets of elements in the numerical semigroup $\Gamma = \<6,9,20\>$ from Example~\ref{e:deltanumerical}.  Here, a dot is placed at the point $(n,d)$ whenever $d \in \Delta(n)$.}  
\label{f:deltanumerical}
\end{figure}

\begin{remark}\label{r:deltaalgorithm}
One major consequence of Theorem~\ref{t:deltahilbert} is an algorithm for computing $\Delta(\Gamma)$ for any finitely generated semigroup $\Gamma$.  
In general, the primary difficulty in computing $\Delta(\Gamma)$ is ensuring that a given value does \emph{not} occur in $\Delta(\Gamma)$.  
Indeed, some elements of $\Delta(\Gamma)$ may only occur in the delta sets of a small finite number of semigroup elements.  For example, if $\Gamma = \<17,33,53,71\>$, then $\Delta(\Gamma) = \{2,4,6\}$, but $6$ is only found in $\Delta(266)$, $\Delta(283)$, and $\Delta(300)$.  

As such, although it is computationally feasible to compute the delta set of any single element of $\Gamma$ (since each has only finitely many factorizations), this cannot be accomplished for all of the (infinitely many) elements of $\Gamma$.  To date, all existing delta set algorithms use some version of Corollary~\ref{c:deltanumerical} to restrict this computation to a finite list of semigroup elements, but consequently all such algorithms are limited to numerical semigroups; see \cite{dynamicalg} for more detail.  

Theorem~\ref{t:deltahilbert} provides the first delta set algorithm for finitely generated semigroups, one which does not rely on computing delta sets of individual semigroup elements.  In particular, computing generators for the ideals in Theorem~\ref{t:deltahilbert} (using \texttt{4ti2} or \texttt{Normaliz}, for instance), together with Gr\"obner basis techniques, yields the delta set of any finitely generated semigroup.  The resulting algorithm is already implemented and will be available in the next release of the \texttt{GAP} package \texttt{numericalsgps} \cite{numericalsgpsgap}, and a discussion of its design and implementation, along with benchmarks, appears in \cite{affineinvariantcomp}.  See also the survey~\cite{compoverview} for an overview of factorization invariant computation.  
\end{remark}

\section{$\omega$-primality}\label{s:omega}


The main result of this section is Theorem~\ref{t:omegahilbert}, which states that the $\omega$-primality invariant (Definition~\ref{d:omega}) is eventually quasilinear over any semigroup $\Gamma \subset A$.  This is proven in two steps: first, we prove that the maximum factorization length function is eventually quasilinear for any such semigroup $\Gamma$ (Theorem~\ref{t:maxlenhilbert}); next, we apply Theorem~\ref{p:omegamaxlen}, which expresses the $\omega$-function of $\Gamma$ in terms of maximum factorization length functions of certain subsemigroups of $\Gamma$.  Specializing Theorems~\ref{t:maxlenhilbert} and~\ref{t:omegahilbert} to numerical semigroups (Corollaries~\ref{c:maxlennumerical} and~\ref{c:omeganumerical}) recovers known results.  

\begin{defn}\label{d:maxlen}
Suppose $\Gamma \subset A$.  The \emph{maximum factorization length} and \emph{minimum factorization length} functions $\mathsf M_\Gamma, \mathsf m_\Gamma:\Gamma \to \NN$ are given by $\mathsf M_\Gamma(\alpha) = \max \mathsf L_\Gamma(\alpha)$ and $\mathsf m_\Gamma(\alpha) = \min \mathsf L_\Gamma(\alpha)$ for each $\alpha \in \Gamma$.  
\end{defn}

We begin by realizing the max factorization length function of any $\Gamma \subset A$ as the Hilbert function of a multigraded module over a graded $R_\Gamma$-algebra (Theorem~\ref{t:maxlenhilbert}).  Corollary~\ref{c:maxlennumerical} examines the case when $\Gamma$ is a numerical semigroup.  An analogous construction yields similar results for the min factorization length function (Corollary~\ref{c:minlenhilbert}).  

\begin{thm}\label{t:maxlenhilbert}
If $\Gamma = \<\alpha_1, \ldots, \alpha_r\> \subset A$, then $\mathsf M_\Gamma:\Gamma \to \NN$ is eventually quasilinear.  
\end{thm}

\begin{proof}
Let 
$$S = R_\Gamma[x_1, x_2]/I_\mathsf L = \kk[x_1, x_2, y_1, \ldots, y_r]/I_\mathsf L$$
with $\deg(x_1) = \deg(x_2) = 0$, and consider the subring 
$$R = \kk[x_1y_1, x_2y_1, \ldots, x_1y_r, x_2y_r] \subset S$$
of $S$.  Since each generator of $R$ has nonzero degree, each graded degree of $R$ has finite dimension over $\kk$.  Let 
$$I = \<x_1^bx_2^c\yy^\aa \in M : |\aa| < \mathsf M_\Gamma(\alpha), \aa \in \mathsf Z_\Gamma(\alpha)\> \subset R.$$
The key observation is that for $\alpha \in \Gamma$, $\aa \in \mathsf Z_\Gamma(\alpha)$ and $b, c \in \NN$, the monomial $x_1^bx_2^c\yy^\aa$ lies in $I$ precisely when $|\aa| < \mathsf M_\Gamma(\alpha)$.  Indeed, if $|\aa| < |\bb|$ for some $\bb \in \mathsf Z_\Gamma(\alpha)$, then $|\aa + \mathbf e_i| < |\bb + \mathbf e_i|$, so the set of monomials corresponding to non-maximal length factorizations is closed under multiplication by monomials in $R$.  

Now, this means for $\aa, \bb \in \mathsf Z(\alpha)$, any two monomials $x_1^bx_2^c\yy^\aa, x_1^{b'}x_2^{c'}\yy^\bb \in R$ with nonzero image modulo $I$ satisfy $|\aa| = |\bb| = \mathsf M_\Gamma(\alpha)$, and thus have equal image precisely when $b = b'$ and $c = c'$ by Theorem~\ref{p:lenhilbert}.  Additionally, each monomial $x_1^bx_2^c\yy^\aa \in R$ satisfies $|\aa| = b + c + 1$.  In particular, for each $\aa \in \NN^r$, $R$ has precisely $|\aa| + 1$ monomials of the form $x_1^bx_2^c\yy^\aa$.  This proves $\mathsf M_\Gamma(\alpha) = \mathcal H(R/I;\alpha) - 1$ for all $\alpha \in \Gamma$, which is eventually quasipolynomial by Theorem~\ref{t:hilbert2}.  

It remains to show that $\mathsf M_\Gamma$ is eventually quasilinear.  Fix $\alpha \in \Gamma$ and a maximal length factorization $\aa \in \mathsf Z_\Gamma(\alpha)$, written as $\alpha = \beta_1 + \cdots + \beta_{|\aa|}$ for $\beta_i \in \{\alpha_1, \ldots, \alpha_r\}$.  By the above argument, $\mathsf M_\Gamma(\beta_1 + \cdots + \beta_i) = i$ for each $i \le |\aa|$.  In particular, $\mathsf M_\Gamma(\alpha) \le |\alpha'|$, where $\alpha' \in \NN^d$ is the projection of $\alpha$ onto $\NN^d$, so $\mathsf M_\Gamma$ grows at most linearly.  Since factorization lengths are unbounded in $\Gamma$, $\mathsf M_\Gamma$ is also unbounded, so we are done.  
\end{proof}

Corollary~\ref{c:maxlennumerical}, as well as the portion of Corollary~\ref{c:minlenhilbert} pertaining to numerical semigroups, appeared as \cite[Theorems~4.2 and~4.3]{elastsets}, respectively.  

\begin{cor}\label{c:maxlennumerical}
If $\Gamma = \<n_1, \ldots, n_r\> \subset \NN$ is a numerical semigroup, then $\mathsf M_\Gamma$ is eventually quasilinear with period dividing $n_1$ and constant leading coefficient $1/n_1$.  
\end{cor}

\begin{proof}
Resume notation from the proof of Theorem~\ref{t:maxlenhilbert}, and write 
$$\mathsf M_\Gamma(n) = a_1(n)n + a_0(n)$$
for periodic functions $a_0,a_1:\NN \to \QQ$ and $n \gg 0$.  Applying Theorem~\ref{t:hilbert}, we wish to show that $(x_1y_1, x_2y_1)$ is a homogeneous system of parameters for $R/I$.  Indeed, $\dim R/I = 2$ by Theorem~\ref{t:maxlenhilbert}, and the quotient $R/\<x_1y_1,x_2y_1\>I$ has finite length.  Now, some element has degree relatively prime to $n_1$ since $\gcd(\Gamma) = 1$, so by Theorem~\ref{t:constcoeffs}, the leading coefficient $a_1$ is constant.  
\end{proof}

\begin{cor}\label{c:minlenhilbert}
Suppose $\Gamma \subset A$.  
The min factorization length function $\mathsf m_\Gamma:\Gamma \to \NN$ is eventually quasilinear.  Moreover, if $\Gamma = \<n_1, \ldots, n_r\> \subset \NN$ is a numerical semigroup, then $\mathsf m_\Gamma$ has period dividing $n_k$ and constant leading coefficient $1/n_k$.  
\end{cor}

Corollary~\ref{c:lenleadingcoeff} refines Theorem~\ref{t:lennumerical}; see Remark~\ref{r:lenleadingcoeff}.  

\begin{cor}\label{c:lenleadingcoeff}
Resuming notation from Theorem~\ref{t:lennumerical}, we have
$$|\mathsf L_\Gamma(n)| = \frac{n_r-n_1}{gn_1n_r}n + a_0(n)$$
for $n \gg 0$, where $g = \min\Delta(\Gamma)$.  
\end{cor}

\begin{proof}
Let $I_\mathsf L = I_0 \subset I_1 \subset I_2 \subset \cdots$ denote the chain of ideals from Theorem~\ref{t:deltahilbert}, and~$J$ denote the defining toric ideal of $\Gamma$.  Both $I_\mathsf L$ and $J$ are prime and $\dim I_\mathsf L = \dim J + 1 = 2$, so since $I_\mathsf L \subsetneq I_j \subset J$ whenever $j \ge \min\Delta(\Gamma)$, we have $\dim I_j = \dim J = \dim I_\mathsf L - 1 = 1$.  As such, $\dim I_g/I_{g-1} = 2$, and $\dim I_j/I_{j-1} = 1$ for $j > g$.  

Now, by Theorems~\ref{t:hilbert} and~\ref{t:deltahilbert}, the number of successive differences equal to $j$ in $\mathsf L(n)$ is eventually periodic if $j > g$.  This implies that for some $n \gg 0$ and $c > 0$, $\mathsf L(n + cn_1n_r)$ has the same number of successive length differences equal to $j$ as $\mathsf L(n)$ for all $j > g$.  As such, by Corollaries~\ref{c:maxlennumerical} and~\ref{c:minlenhilbert} we have 
$$\begin{array}{r@{}c@{}l}
|\mathsf L(n + cn_1n_r)| - |\mathsf L(n)|
&{}={}& \textstyle\frac{1}{g}\big((\mathsf M_\Gamma(n + cn_1n_r) - \mathsf m_\Gamma(n + cn_1n_r)) - (\mathsf M_\Gamma(n) - \mathsf m_\Gamma(n))\big) \\[0.05in]
&{}={}& \frac{1}{g}(cn_r - cn_1),
\end{array}$$
which implies the leading coefficient $a_1$ has the desired form.  
\end{proof}

In the remainder of this section, we use Theorem~\ref{t:maxlenhilbert} to show that the $\omega$-primality invariant (Definition~\ref{d:omega}) is eventually quasilinear over any affine semigroup.  See~\cite{omegamonthly} for a more thorough introduction to $\omega$-primality.  

\begin{defn}\label{d:omega}
Suppose $\Gamma = \<\alpha_1, \ldots, \alpha_r\> \subset A$.  
For each $\alpha \in \Gamma$, define $\omega(\alpha) = m$ if $m$~is the smallest positive integer with the property that whenever $a_1\alpha_1 + \cdots + a_r\alpha_r - \alpha \in \Gamma$ for some $\aa \in \NN^r$, there is a $\bb \in \NN^r$ satisfying $|\bb| \leq m$ and $b_i \le a_i$ for each $i \le r$ such that $b_1\alpha_1 + \cdots + b_r\alpha_r - \alpha \in \Gamma$.  
\end{defn}

In the remainder of this section, we prove the $\omega$-function is eventually quasilinear for any semigroup $\Gamma \subset A$ (Theorem~\ref{t:omegahilbert}).  This is done by combining Theorem~\ref{t:maxlenhilbert} and Lemmas~\ref{l:maxquasilinear}-\ref{l:maxconequasilinear} with the following characterization of $\omega$-primality, which also appeared as \cite[Theorem~6.1]{dynamicalg} for numerical semigroups.  

\begin{prop}\label{p:omegamaxlen}
Suppose $\Gamma = \<G\> \subset A$ for $G = \{\alpha_1, \ldots, \alpha_r\}$.  
For $T \subset G$, define 
$$Ap(T) = \{\alpha \in \Gamma : \alpha - \alpha_i \notin \Gamma \text{ for all } \alpha_i \in T\}.$$
We have 
$$\omega(\alpha) = \max\left\{\mathsf M_{\<T\>}(\alpha + \beta) : \emptyset \ne T \subset G \text{ and } \beta \in Ap(T) \right\}$$
for all $\alpha \in \Gamma$.  
\end{prop}

\begin{proof}
By \cite[Proposition~2.10]{omegamonthly}, $\omega(\alpha)$ is the maximum value of $b_1 + \cdots + b_r$ among $\bb \in \NN^r$ satisfying (i) $b_1\alpha_1 + \cdots + b_r\alpha_r - \alpha \in \Gamma$, and (ii) $b_1\alpha_1 + \cdots + b_r\alpha_r - \alpha - \alpha_i \notin \Gamma$ for each $i$ with $b_i > 0$.  Notice that each $\bb \in \NN^r$ satisfying (i) gives a factorization of $\alpha + \beta$ in $\<T\>$, where $\beta = b_1\alpha_1 + \cdots + b_r\alpha_r - \alpha$ and $T = \{\alpha_i : b_i > 0\}$.  Additionally, $\bb$~satisfies condition (ii) if and only if $\beta$ lies in $Ap(T)$.  Thus, $\omega(\alpha)$ is the maximal length of all such factorizations~$\bb$, as desired.  
\end{proof}

\begin{lemma}\label{l:maxquasilinear}
The maximum of finitely many eventually quasilinear functions on $A$ is eventually quasilinear.  
\end{lemma}

\begin{proof}
By induction, it suffices to prove that $\max(f,g)$ is eventually quasilinear for any two eventually quasilinear functions $f, g:A \to \QQ$.  Applying Theorem~\ref{t:affineequiv}, it suffices to assume $f$ and $g$ are simple quasilinear functions supported on the same cone $C$, which by appropriate translation we can assume is based at $0 \in A$.  By Lemma~\ref{l:projection}, we can assume $A = \NN^d$.  We have $\max(f,g) = f$ precisely when $f - g$ is non-negative, and since $f$ and $g$ each coincide with a rational linear function, this happens on a rational linear halfspace $H \subset \NN^d$.  The semigroup $C \cap H$ is finitely generated by Gordan's Lemma \cite[Theorem~7.16]{cca}, and thus is a disjoint union of finitely many cones.   
\end{proof}


\begin{lemma}\label{l:aperystrat}
Suppose $\Gamma \subset A$, and fix $\beta \in A$.  The set 
$$\{\alpha \in \Gamma : \alpha - \beta \notin \Gamma\} \subset \Gamma$$
is a finite union of disjoint cones.  
\end{lemma}

\begin{proof}
Let $R = \kk[\xx^\alpha : \alpha \in \Gamma] \subset \kk[A]$ with $\deg(\xx^\alpha) = \alpha$, and let $I = \<\xx^\alpha : \alpha - \beta \in \Gamma\>$.  Notice that $\xx^\alpha \notin I$ whenever $\alpha - \beta \notin \Gamma$, so 
$$\mathcal H(R/I;\alpha) = \left\{\begin{array}{ll}
1 & \alpha - \beta \notin \Gamma \\
0 & \text{otherwise}
\end{array}\right.$$
for any $\alpha \in \Gamma$.  The claim now follows from Theorems~\ref{t:affineequiv} and~\ref{t:hilbert2}.  
\end{proof}

\begin{lemma}\label{l:maxconequasilinear}
Fix $f:A \to \QQ$ eventually quasilinear, and fix $\alpha_1, \ldots, \alpha_r \in A$.  Let 
$$F(\alpha) = \max\{f(\alpha + a_1\alpha_1 + \cdots + a_r\alpha_r) : a_1, \ldots, a_r \in \NN\},$$
and assume $F(\alpha)$ is finite for all $\alpha \in \Gamma$.  Then $F$ is eventually quasilinear.  
\end{lemma}

\begin{proof}
By Theorem~\ref{t:affineequiv}, it suffices to assume $f$ is simple quasilinear.  Let $C \subset A$ denote the cone on which $f$ is supported.  Considering each $\alpha_i$ in turn in what follows, it suffices to assume $r = 1$.  Since $f$ is linear, there exists a constant $q \in \QQ$ such that $f(\alpha + \alpha_1) - f(\alpha) = q$ for all $\alpha \in A$.  If $q \le 0$, then $F(\alpha) = f(\alpha)$ for all $\alpha$.  If, on the other hand, $q > 0$, then the set $\{\alpha + m\alpha_1 : m \ge 0\} \cap C$ is finite for all $\alpha$ since each $F(\alpha)$ is finite.  In particular, if $m_\alpha \in \NN$ is maximal with the property that $\alpha + m_\alpha \alpha_1 \in C$, then $F(\alpha) = f(\alpha + m_\alpha\alpha_1)$.  By Lemma~\ref{l:aperystrat}, $\{\alpha + m_\alpha \alpha_1 : \alpha \in C\}$ is a finite union of disjoint cones $C_1, \ldots, C_k$.  Partition $C$ into sets $P_1, \ldots, P_k$ with $P_i = \{\alpha : \alpha + m_\alpha \alpha_1 \in C_i\}$, and observe that $F(\alpha)$ equals the projection of $f$ onto $C_i$ whenever $\alpha \in P_i$.  
\end{proof}

\begin{thm}\label{t:omegahilbert}
The $\omega$-function on any $\Gamma = \<\alpha_1, \ldots, \alpha_r\> \subset A$ is eventually quasilinear.  
\end{thm}

\begin{proof}
Fix a nonempty subset $T \subset \{\alpha_1, \ldots, \alpha_r\}$.  By Theorem~\ref{t:maxlenhilbert}, $\mathsf M_{\<T\>}$ is eventually quasilinear.  Using the notation from Proposition~\ref{p:omegamaxlen}, $Ap(T)$ is a finite union of disjoint cones by Lemma~\ref{l:aperystrat}, and for each cone $C$, the map $\alpha \mapsto \max\{M_{\<T\>}(\alpha + \beta) : \beta \in C\}$ is eventually quasilinear by Lemma~\ref{l:maxconequasilinear}.  Lastly, taking the maximum over all nonempty subsets $T$ of $\{\alpha_1, \ldots, \alpha_r\}$ completes the proof by Lemma~\ref{l:maxquasilinear}.  
\end{proof}

Upon specializing Theorem~\ref{t:omegahilbert} to numerical semigroups, we obtain Corollary~\ref{c:omeganumerical}, which appeared as \cite[Theorem~3.6]{omegaquasi} and \cite[Corollary~20]{compasympomega}.  

\begin{cor}\label{c:omeganumerical}
Fix a numerical semigroup $\Gamma = \<n_1, \ldots, n_r\> \subset \NN$.  The $\omega$-function on $\Gamma$ is eventually quasilinear with period $n_1$ and constant leading coefficient $1/n_1$.  
\end{cor}

\begin{proof}
Specializing the proof of Theorem~\ref{t:omegahilbert} to numerical semigroups proves $\omega$ is quasilinear.  Additionally, resuming the notation from Proposition~\ref{p:omegamaxlen}, the set $Ap(T)$ is finite for each $T \subset \{n_1, \ldots, n_r\}$.  Since each function $\mathsf M_{\<T\>}$ is quasilinear with constant linear coefficient $1/\min T$ by Theorem~\ref{t:maxlenhilbert}, those with $\min T = n_1$ will eventually dominate.  Each such function also has period $n_1$ by Theorem~\ref{t:maxlenhilbert}, as desired.  
\end{proof}

\begin{remark}\label{r:omegadirect}
Although this section provides a proof of Theorem~\ref{t:omegahilbert}, the argument requires carefully combining (in general infinitely many) Hilbert functions.  It remains an interesting problem to construct a single graded module (or at least finitely many) whose Hilbert function(s) determine the $\omega$-function for a given semigroup, as this would prove Theorem~\ref{t:omegahilbert} using a more direct application of Theorem~\ref{t:hilbert2}.  
\end{remark}

\begin{prob}\label{pb:omegadirect}
Realize the $\omega$-function on $\Gamma \subset A$ as a Hilbert function directly, without appealing to Theorem~\ref{t:maxlenhilbert}.  
\end{prob}

\section{The catenary degree}\label{s:catenary}

The final factorization invariant considered in this paper is the catenary degree (Definition~\ref{d:catenary}).  As with the delta set invariant in Section~\ref{s:delta}, a family of modules whose Hilbert functions determine the catenary degree is constructed (Theorem~\ref{t:catenaryhilbert}).  Applying Hilbert's theorem classifies the eventual behavior of the catenary degree (Corollary~\ref{c:catenaryhilbert}) and specializes to a known result for numerical semigroups (Corollary~\ref{c:catenarynumerical}).

\begin{defn}\label{d:catenary}
Fix $\alpha \in \Gamma = \<\alpha_1, \ldots, \alpha_r\> \subset A$.  
For $\aa, \bb \in \mathsf Z_\Gamma(\alpha)$, the \emph{greatest common divisor of $\aa$ and $\bb$} is given by 
$$\gcd(\aa,\bb) = (\min(a_1,b_1), \ldots, \min(a_r,b_r)) \in \NN^r,$$
and the \emph{distance between $\aa$ and $\bb$} (or the \emph{weight of $(\aa,\bb)$}) is given by 
$$d(\aa,\bb) = \max(|\aa - \gcd(\aa,\bb)|,|\bb - \gcd(\aa,\bb)|).$$
Given $\aa, \bb \in \mathsf Z_\Gamma(\alpha)$ and $N \ge 1$, an \emph{$N$-chain from $\aa$ to $\bb$} is a sequence $\aa_1, \ldots, \aa_k \in \mathsf Z_\Gamma(\alpha)$ of factorizations of $\alpha$ such that (i) $\aa_1 = \aa$, (ii) $\aa_k = \bb$, and (iii) $d(\aa_{i-1},\aa_i) \le N$ for all $i \le k$.  The \emph{catenary degree of $\alpha$}, denoted $\mathsf c(\alpha)$, is the smallest non-negative integer $N$ such that there exists an $N$-chain between any two factorizations of $\alpha$.  
\end{defn}

In the proof of Theorem~\ref{t:catenaryhilbert}, we use an equivalent characterization of the catenary degree presented in Proposition~\ref{p:catenaryequiv}.  

\begin{defn}\label{d:redundantpair}
Fix $\alpha \in \Gamma = \<\alpha_1, \ldots, \alpha_r\> \subset A$.  
A pair $(\aa,\bb)$ of factorizations of $\alpha$ is \emph{redundant} if there exists a $N$-chain from $\aa$ to $\bb$ for some $N < d(\aa,\bb)$.  
\end{defn}

\begin{prop}\label{p:catenaryequiv}
Suppose $\Gamma \subset A$.  The catenary degree of $\alpha \in \Gamma$ is
$$\mathsf c(\alpha) = \max\{d(\aa,\bb) : (\aa,\bb) \text{ not redundant}\},$$
that is, the maximal weight of a non-redundant pair of factorizations.  
\end{prop}

\begin{proof}
Any pair $(\aa,\bb)$ of factorizations of $\alpha$ with $d(\aa,\bb) > \mathsf c(\alpha)$ is redundant.  Moreover, minimality of $\mathsf c(\alpha)$ ensures there exists a non-redundant pair $(\aa,\bb)$ of factorizations of $\alpha$ with $d(\aa,\bb) = \mathsf c(\alpha)$.  
\end{proof}

\begin{thm}\label{t:catenaryhilbert}
Suppose $\Gamma = \<\alpha_1, \ldots, \alpha_r\> \subset A$.  
There is a sequence $M_2, M_3, M_4, \ldots$ of finitely generated, modestly $A$-graded modules such that $\mathcal H(M_j;\alpha) > 0$ if and only if $\alpha$~has a non-redundant pair $(\aa,\bb)$ of factorizations with $d(\aa,\bb) = j$.  In particular, 
$$\mathsf c(\alpha) = \max\{j : \mathcal H(M_j;\alpha) > 0\}.$$
\end{thm}

\begin{proof}
Let $S = \kk[x_1, \ldots, x_r, y_1, \ldots, y_r]$ with $\deg(x_i) = \alpha_i$ and $\deg(y_i) = 0$ for $i \le r$.  Consider the subring 
$$R = \kk[x_1y_1, \ldots, x_ry_r] \subset S,$$
and the $R$-modules $M' \subset M \subset S$ given by
$$\begin{array}{r@{}c@{}l}
M &{}={}& \<\xx^\aa\yy^\bb : \aa, \bb \in \mathsf Z_\Gamma(\alpha), \alpha \in \Gamma\> \text{ and } \\
M' &{}={}& \<\xx^\aa\yy^\bb \in M: (\aa,\bb) \text{ redundant}\>.
\end{array}$$
Notice that each monomial $\xx^\aa\yy^\bb \in M$ corresponds to a pair of factorizations $(\aa,\bb)$ of the element $\alpha = \deg(\xx^\aa\yy^\bb) \in \Gamma$.  

First, we claim $M$ is minimally generated by $\{\xx^\aa\yy^\bb : \gcd(\aa,\bb) = \mathbf 0\}$.  
Indeed, if $a_i > 0$ and $b_i > 0$ for some $i \le r$, then 
$$\xx^\aa\yy^\bb = (x_iy_i)\xx^{\aa-\mathbf e_i}\yy^{\bb-\mathbf e_i} \in M, $$
so $\xx^\aa\yy^\bb$ can be omitted from any monomial generating set for $M$.  

Next, we claim a monomial $\xx^\aa\yy^\bb \in M$ lies in $M'$ if and only if $(\aa,\bb)$ is redundant.  Indeed, $d(\aa+\cc,\bb+\cc) = d(\aa,\bb)$ for all $\aa, \bb, \cc \in \NN^r$, so if $\aa = \aa_1, \ldots, \aa_k = \bb$ is an $N$-chain for $(\aa,\bb)$, then $\aa_1+\cc, \ldots, \aa_k+\cc$ is an $N$-chain for $(\aa + \cc, \bb + \cc)$.  As such, if~$(\aa,\bb)$ is redundant and $\cc \in \NN^r$, then $(\aa + \cc, \bb + \cc)$ is also redundant.  

Now, for each $j \ge 2$, let $M_j$ denote the $R$-submodule of $M/M'$ given by
$$M_j = \<\xx^\aa\yy^\bb \in M : d(\aa,\bb) = j\> \subset M/M'.$$
By the above argument, every monomial $\xx^\aa\yy^\bb \in M_j$ satisfies $d(\aa,\bb) = j$.  As such, we conclude $M_j$ has a monomial of degree $\alpha$ precisely when $\alpha$ has a non-redundant pair of factorizations with weight $j$.  This implies $\mathsf c(\alpha)$ has the desired form by Proposition~\ref{p:catenaryequiv}.  

It remains to show that each $M_j$ is finitely generated.  If $\aa, \bb \in \mathsf Z_\Gamma(\alpha)$ satisfy $\gcd(\aa,\bb) = \mathbf 0$, then $d(\aa,\bb) = \max(|\aa|, |\bb|)$, so only finitely many such pairs can also satisfy $d(\aa,\bb) = j$.  Since $M_j$ is generated by those monomials $\xx^\aa\yy^\bb$ in the minimal generating set of $M$ satisfying $d(\aa,\bb) = j$, this completes the proof.  
\end{proof}

Applying Theorems~\ref{t:affineequiv}(c) and~\ref{t:hilbert2} to Theorem~\ref{t:catenaryhilbert} yields Corollary~\ref{c:catenaryhilbert}.  

\begin{cor}\label{c:catenaryhilbert}
Suppose $\Gamma = \<\alpha_1, \ldots, \alpha_r\> \subset A$.  
For each $j \ge 2$, the set 
$$\{\alpha \in \Gamma : \mathsf c(\alpha) = j\}$$
is a finite union of disjoint cones.  In particular, the catenary degree function $\mathsf c:\Gamma \to \NN$ is eventually quasiconstant.  
\end{cor}

Specializing Corollary~\ref{c:catenaryhilbert} to numerical semigroups yields Corollary~\ref{c:catenarynumerical}, which appeared as \cite[Theorem~3.1]{catenaryperiodic}.  

\begin{cor}\label{c:catenarynumerical}
Fix a numerical semigroup $\Gamma = \<n_1, \ldots, n_r\> \subset \NN$.  The catenary degree function $\mathsf c:\Gamma \to \ZZ_{\ge 0}$ is eventually periodic, and its period divides $\lcm(n_1, \ldots, n_r)$.  
\end{cor}

\begin{proof}
Eventual periodicity follows from Theorem~\ref{t:catenaryhilbert}.  Resuming the notation from Theorem~\ref{t:catenaryhilbert}, the sequence $(x_1y_1, \ldots, x_ry_r)$ forms a homogeneous system of parameters for each $M_j$, so $\mathsf c$ has period dividing $\lcm(n_1, \ldots, n_r)$ by Theorem~\ref{t:hilbert}.  
\end{proof}

\begin{remark}
The catenary degree is just one of many factorization invariants defined using chains of factorizations.  Many of these other invariants are also known to be eventually periodic for numerical semigroups, and an answer to Problem~\ref{pb:catenaryvariants} would extend these results in the same manner as Theorem~\ref{t:catenaryhilbert}.  See~\cite{halffactorial} for precise definitions.  
\end{remark}

\begin{prob}\label{pb:catenaryvariants}
Generalize Theorem~\ref{t:catenaryhilbert} to describe the monotone catenary degree, homogeneous catenary degree, equal catenary degree, and tame degree.
\end{prob}

\end{document}